\documentclass[a4paper,11pt]{article}
\usepackage{amsmath,amsthm,amssymb,amsfonts}
\usepackage{mathtools} % loads the amsmath package automatically.
\usepackage{bm}% used to get bold symbos; command *\bm{xxxx}*T
\usepackage{mathrsfs} % try the command *$\mathscr{H}$*
\usepackage{changes} % Used to make changes in a text, e.g., with the command *\deleted{xxxx}*
\usepackage{url} % used to write *email adresses*, *hypertexts links*, etc.
\usepackage[toc,page]{appendix} %-> a opção ``page'' permite criar um ``environment'' e a opção ``toc'' 
\usepackage{enumerate}%The package adds an optional argument to the enumerate environment which determines the style in 
\usepackage{enumitem}% To customize the label in itemize

\usepackage{cases}% To use separate labels in the "cases" enviorment

\usepackage{graphicx} %\usepackage[dvips]{graphicx}  %\usepackage[pdftex]{graphicx}
\usepackage{longtable} %Allow tables to flow over page boundaries
\usepackage[all,knot,arc,import,poly]{xy}
\usepackage{newfloat} %The package offers the command \DeclareFloatingEnvironment
\DeclareFloatingEnvironment[fileext=frm,placement={ht},name=Frame]{algfloat}
\usepackage{caption}%provides many ways to customise the captions in floating environments like figure and table
\usepackage{placeins}%Defines a \FloatBarrier command, beyond which floats may not pass; useful, for example, to ensure all %floats for a section appear before the next \section command
\usepackage{tikz}%TikZ is probably the most complex and powerful tool to create graphic elements in LaTeX.
\usepackage{subfig}%The package provides support for the manipulation and reference of small or ‘sub’ figures and tables within a single figure or table environment
\usepackage{pgfplots}%PGFPlots draws high-quality function plots in normal or logarithmic scaling with a user-friendly interface directly in TeX
\usepackage{color}
\usepackage{xcolor} %\usepackage[usenames,dvipsnames,svgnames,table]{xcolor}%The package starts from the basic facilities %
\usepackage{array}%An extended implementation of the array and tabular environments which extends the options for column formats
\usepackage{booktabs}%The package enhances the quality of tables in LaTeX, providing extra commands as well as behind-the-scenes optimisation
\usepackage{grffile}%The original package extended the file name processing of package graphics to support a larger range of file names. 
\usetikzlibrary{patterns}
\usepackage{tabularx}
\usepackage{booktabs}
\usepackage{siunitx}
\usepackage{makecell}
\usepackage{adjustbox}
\usepackage{changepage}   % adjustwidth
\usepackage{booktabs,siunitx,makecell,adjustbox}
\newcommand{\N}{\mathbb{N}}
\newcommand{\R}{\mathbb{R}}
\newcommand{\BR}{\R \cup \set{+\infty}}
\newcommand{\inner}[2]{\langle #1\,,\,#2\rangle} % [2] number of arguments. Modo de uso: \inner{a}{b}= <a,b>
\newcommand{\norm}[1]{\|{#1}\|}
\newcommand{\normq}[1]{\|{#1}\|^2}

\newcommand{\Ha}{{\mathcal{H}}} % Espaço de Hilbert
\newcommand{\HH}{{\mathcal{H}}} % Espaço de Hilbert
\newcommand{\tos}{\rightrightarrows} % \to is already defined in LaTex
\newcommand{\wto}{\rightharpoonup}

\newcommand{\comenta}[1]{} % comentar textos - no texworks pode ser feito com CTRL + SHFIT + ]
\newcommand{\set}[1]{\{#1\}}
\newcommand{\Set}[1]{\left\{#1\right\}}

\newcommand{\lab}[1]{\label{#1}}

\newcommand{\bprop}{\begin{proposition}}
\newcommand{\eprop}{\end{proposition}}
\newcommand{\blemm}{\begin{lemma}}
\newcommand{\elemm}{\end{lemma}}
\newcommand{\bdefi}{\begin{definition}}
\newcommand{\edefi}{\end{definition}}
\newcommand{\btheo}{\begin{theorem}}
\newcommand{\etheo}{\end{theorem}}
\newcommand{\brema}{\begin{remark}}
\newcommand{\erema}{\end{remark}}
\newcommand{\bassu}{\begin{assumption}}
\newcommand{\eassu}{\end{assumption}}
\newcommand{\bcoro}{\begin{corollary}}
\newcommand{\ecoro}{\end{corollary}}

\newcommand{\benum}{\begin{enumerate}[label = \emph{(\alph*)}]}
\newcommand{\eenum}{\end{enumerate}}

\newcommand{\mgap}{\vspace{.1in}}

\newcommand{\beq}{\begin{equation}}
\newcommand{\eeq}{\end{equation}}
\DeclareMathOperator{\Dom}{Dom} % \newcommand is in general used for much more complex situations

\DeclareMathOperator{\Gra}{Gra}
\DeclareMathOperator{\dom}{dom}
\DeclareMathOperator{\epi}{epi}

\newtheorem{theorem}{Theorem}[section] %[section] restarts the theorem counter at every new section
\newtheorem{lemma}[theorem]{Lemma}
\newtheorem{corollary}[theorem]{Corollary}
\newtheorem{proposition}[theorem]{Proposition}
\newtheorem{remark}[theorem]{Remark}
\newtheorem{definition}[theorem]{Definition}
\newtheorem{assumption}[theorem]{Assumption}
\usepackage[ruled]{algorithm2e}
\SetKwInput{KwInput}{Input}
\usepackage[T1]{fontenc}
\usepackage[utf8]{inputenc}
\usepackage[portuguese,english]{babel} %\usepackage[portuguese,german]{babel}

\newcommand{\tiw}{{\widetilde w}}

\makeatletter
\def\widebreve#1{\mathop{\vbox{\m@th\ialign{##\crcr\noalign{\kern\p@}%
  \brevefill\crcr\noalign{\kern0.1\p@\nointerlineskip}%
  $\hfil\displaystyle{#1}\hfil$\crcr}}}\limits}

\def\brevefill{$\m@th \setbox\z@\hbox{}%
 \hfill\scalebox{0.7}{\rotatebox[origin=c]{90}{(}} \kern4pt $}
\makeatletter

\definechangesauthor[name={M. Marques Alves}, color=blue]{M}
\definechangesauthor[name={Raul}, color=green!50!black]{R}
\definechangesauthor[name={Juan}, color=green!50!black]{J}

\begin{document}
\title{A strongly convergent inertial inexact proximal-point algorithm for monotone inclusions with applications to variational inequalities}
\author{ 
M. Marques Alves
\thanks{ Departamento de Matem\'atica,
Universidade Federal de Santa Catarina,
Florian\'opolis, Brazil, 88040-900 ({\tt maicon.alves@ufsc.br}).
The work of this author was partially supported by CNPq grant 308036/2021-2 and Fundação de Amparo a
Pesquisa e Inovação do Estado de Santa Catarina (FAPESC), Edital 21/2024, Grant 2024TR002238.}
\and
J. E. Navarro Caballero
\thanks{ Departamento de Matem\'atica,
Universidade Federal de Santa Catarina,
Florian\'opolis, Brazil, 88040-900 ({\tt juan.caballero@ufsc.br}). The work of this author was partially supported by CAPES 
scholarship 88887.829718/2023-00.}
\and
M. Geremia
\thanks{ Departamento de Ensino, Pesquisa e Extensão, Instituto Federal de Santa Catarina (IFSC),
Chapecó, Brazil, 89813-000 ({\tt marina.geremia@ifsc.edu.br}).}
\and
R. T. Marcavillaca
\thanks{ Centro de Modelamiento Matem\'atico (CNRS UMI2807), Universidad de Chile, Santiago, Chile, 8370458 ({\tt raultm.rt@gmail.com; rtintaya@dim.uchile.cl}). The work of this author was partially supported by  BASAL fund FB210005 for center of excellence from ANID-Chile.
}
}
%\date{May 9, 2014}
\date{\emph{\small Dedicated to Juan Enrique Martínez-Legaz on the occasion of his 70th birthday}}

\maketitle
%\tableofcontents
\pagestyle{plain}

\begin{abstract}
We propose an inertial variant of the strongly convergent inexact proximal-point (PP) method of Solodov and Svaiter (2000) for monotone inclusions. 
We prove strong convergence of our main algorithm under less restrictive assumptions on the inertial parameters when compared to previous analysis of inertial PP-type algorithms, which makes our method of interest even in finite-dimensional settings.
We also performed an iteration-complexity analysis and applied our main algorithm to variational inequalities for monotone operators, obtaining strongly convergent (inertial) variants of Korpolevich's extragradient, forward-backward and Tseng's modified forward-backward methods.  
Preliminary numerical experiments indicate that our strongly convergent variant of Tseng's modified forward-backward method performs well on certain matrix game problems.
\\
\\ 
2000 Mathematics Subject Classification: 47H05, 47N10, 90C33.
%
%90C25 (Convex Programming)
%90C06 (Large Scale Problems in Mathematical Programming)
%
%47H05 (Monotone Operators and Generalizations).
%47N10 (Applications of operator theory in optimization, convex analysis, mathematical programming, economics)
%47J25  (Iterative procedures involving nonlinear operators)
%    
%49J52 (Nonsmooth Analysis)
%49M15 (Newton-type methods)
%49M27 (Decomposition methods)
%
%68Q25 (Analysis of algorithms and problem complexity)
%
%90C30 (Nonlinear Programming)
%90C33 (Complementarity and equilibrium problems and variational inequalities)
\end{abstract}

\pagestyle{plain}

\section{Introduction} \lab{sec:int}

We consider the \emph{monotone inclusion problem}
\begin{align} \lab{eq:prob}
0 \in T(x)
\end{align}
where $T \colon \Ha \tos \Ha$ is a (set-valued) maximal monotone operator as well as the (structured) inclusions
\begin{align}\lab{eq:prob02}
0 \in F(x) + B(x)
\end{align}
and
\begin{align}\lab{eq:prob03}
0 \in F(x) + N_C(x)
\end{align} 
where $F(\cdot)$ is a (single-valued) monotone map, $B(\cdot)$ is (set-valued) maximal monotone and 
$N_C(\cdot)$ is the normal cone of the convex set $C$ (more details on notation will be given in Subsection \ref{subsec:gn} below). Under mild assumptions on the operators $F(\cdot)$ and $B(\cdot)$, and the set $C$, problems \eqref{eq:prob02} and \eqref{eq:prob03} are instances of the general problem \eqref{eq:prob}, where $T = F + B$ and $T = F + N_C$, respectively.
Problem \eqref{eq:prob} and its structured versions \eqref{eq:prob02} and \eqref{eq:prob03} appear in a wide range of applications, including convex optimization, economics, machine learning, among others (see, e.g.,~\cite{BC11} and references therein).
Using the definition of $N_C(\cdot)$, it is simple to check that \eqref{eq:prob03} is also equivalent to the problem of
finding $x\in C$ satisfying
\begin{align}\label{eq:prob04}
\inner{F(x)}{y - x}\geq 0\quad \mbox{for all}\;\;y\in C,
\end{align}
which is traditionally known as the (Stampacchia)  \emph{variational inequality problem} (VIP) for $F(\cdot)$ and $C$. 
Since in this paper the (monotone inclusion) formulation \eqref{eq:prob03} is more convenient, from now on we will also refer to 
\eqref{eq:prob03} as a VIP.

Our focus will be on solving \eqref{eq:prob}, \eqref{eq:prob02} and \eqref{eq:prob03} numerically. Specifically, we aim to design and analyze numerical schemes capable of computing approximate solutions to these problems.
Regarding \eqref{eq:prob}, the most popular algorithm for computing its approximate solutions is the 
\emph{proximal-point} (PP) 
\emph{method}, which can be described as follows: given $x_0\in \Ha$, define the sequence $\set{x_k}$ iteratively as
\begin{align}\label{eq:ppm}
x_{k+1} \approx (\lambda_k T + I)^{-1}x_k\quad \mbox{for all}\quad  k\geq 0,
\end{align}
where $\lambda_k > 0$. In \eqref{eq:ppm}, ``$\approx$'' means that $x_{k+1}$ is an approximation to the (exact)
 iterate $(\lambda_k T + I)^{-1}x_k$, defining, in this way, inexact versions of the PP method.
In the seminal paper~\cite{Roc76}, Rockafellar studied the \emph{weak} convergence of \eqref{eq:ppm} under summable error criteria and applied the obtained results to the justification of augmented Lagrangian algorithms for convex programming.
The \emph{strong} convergence of sequence $\set{x_k}$ as in \eqref{eq:ppm} was left as an open problem, later on solved negatively by Güller in~\cite{gul-new.sjo92}.  

The problem of forcing the strong convergence of the PP algorithm \eqref{eq:ppm} (by conveniently modifying 
the iterative process) was addressed by Solodov and Svaiter~\cite{SS00} in the setting of relative-error inexact PP methods,  
which we briefly discuss next. Relative-error inexact PP-type methods appeared as an alternative to proximal algorithms employing summable error criteria for subproblems; the first methods of this type were proposed in~\cite{sol.sva-hpe.svva99, sol.sva-hyb.jca99} and subsequently studied, e.g., in~\cite{AMS16, MS10, MS12, MS13}.
The key idea consists in observing that the exact PP iteration $x_+ = (\lambda T + I)^{-1}x$ can be decoupled as an
inclusion-equation system
\begin{align} \lab{eq:dec.prox}
  v \in T(x_+),\quad  \lambda v + x_+ - x = 0,
\end{align}
and then relaxing \eqref{eq:dec.prox} within relative-error tolerance criteria -- see Definition \ref{def:ine-sol} 
and Eq. \eqref{eq:wood} below.

The starting point for this contribution is the paper~\cite{SS00}: we will introduce \emph{inertial effects} on the iteration of~\cite[Algorithm 1]{SS00}, obtaining in this way a strongly convergent inertial (relative-error) inexact PP method for solving \eqref{eq:prob}. We will also study the iteration-complexity of our algorithm (see Algorithm \ref{alg:main} below) and apply the results of strong convergence and iteration-complexity to \eqref{eq:prob02} and \eqref{eq:prob03}, through the derivation of
strong convergent inertial variants of the Korpolevich extragradient method, for solving \eqref{eq:prob03}, as well as of the Tseng's modified forward-backward and forward-backward methods for \eqref{eq:prob02}. 

Proximal algorithms with inertial effects for monotone inclusions and related problems were first proposed in the seminal paper~\cite{AA01} and subsequently developed in various research directions also by other authors (see, e.g., 
\cite{AEGM20, Att21, AC20, ACCR18, AP19,bot.acm2019, CG17} and references therein).
The main idea consists in at a current iterate, say $x_k$, introduce an ``inertial effect'' by extrapolating $x_k$ in the direction 
$x_k - x_{k-1}$:
\[
 w_k = x_k + \alpha_k (x_k - x_{k-1}),
\]
where $\alpha_k\geq 0$, and then update the current information from $w_k$ instead of $x_k$ -- see \eqref{eq:defw} and \eqref{eq:wood} below. 
As it was mentioned earlier, our algorithms will benefit from inertial effects on the iteration (see also the comments and remarks following Algorithm \ref{alg:main} for more details). 

\mgap

\noindent
{\bf Main contributions.} Our main results can be summarized as follows:
\begin{itemize}
\item[(i)] Algorithm \ref{alg:main} for solving \eqref{eq:prob} is an inertial modification of~\cite[Algorithm 1]{SS00} with a somehow more general relative-error criterion for subproblems. Its strong convergence and iteration-complexity analyzes are presented in Theorems \ref{th:sconv} and \ref{th:pcom}, respectively. We also refer the reader to the several comments and remarks following Algorithm \ref{alg:main} for additional details.
\item[(ii)] Algorithm \ref{alg:korp} is a strongly convergent inertial version of the famous Korpolevich's extragradient method for VIPs. The main results regarding Algorithm \ref{alg:korp} are summarized in Theorem \ref{th:korp} below. In the same way as in the previous item, we also refer to the comments and remarks following Algorithm \ref{th:korp} for more details.
\item[(iii)] Algorithms \ref{alg:tseng} and \ref{alg:fb} in Section \ref{sec:tseng} are strongly convergent inertial variants of the Tseng's forward-backward and forward-backward methods for solving \eqref{eq:prob02}, respectively.  The main results of this section are summarized in Theorems \ref{th:tseng} and \ref{th:fb.main}.
\end{itemize} 

\mgap

\noindent
{\bf Most related works.} In~\cite{SS00}, by introducing a simple modification of the (relative-error inexact) PP algorithm for monotone inclusions, Solodov and Svaiter obtained a strongly convergent inexact PP method for \eqref{eq:prob}.
In~\cite{BC01}, Bauschke and Combettes introduced and studied a general framework for forcing the strong convergence of 
(weakly convergent) Fej\'er-monotone approximation methods.
In~\cite{DJCS17}, Dong \emph{et al.}  proposed and studied an inertial forward-backward splitting algorithm with strong convergence guarantees for \eqref{eq:prob02} (assuming that $F(\cdot)$ is cocoercive) and applied the obtained results to convex optimization and split feasibility problems.

\mgap

\noindent
{\bf Organization of the paper.} In Section \ref{sec:basic}, we present some basic and preliminary results. 
In Section \ref{sec:main}, we state and study the strong convergence and iteration-complexity of our main algorithm for solving \eqref{eq:prob}, namely Algorithm \ref{alg:main}.
In Section \ref{sec:korp}, we develop a strongly convergent inertial version of the Korpolevich's extragradient method for solving 
VIPs with monotone and Lipschitz continuous operators. 
In Section \ref{sec:tseng}, we present variants of the Tseng's modified forward-backward and forward-backward methods for solving the structured inclusion \eqref{eq:prob02}.
In Section \ref{sec:ne}, we present preliminary numerical experiments on matrix game problems.

\subsection{General notation} \lab{subsec:gn}
Throughout this paper $\Ha$ denotes a (possibly infinite-dimensional) real Hilbert space with inner product 
$\inner{}{}$ and induced norm $\norm{\cdot} = \sqrt{\inner{\cdot}{\cdot}}$.
For a set-valued map $T \colon \Ha \tos \Ha$, the \emph{effective domain} and \emph{graph} of $T$ are 
$\Dom T = \set{x  \mid T(x) \neq \emptyset}$ and $\Gra T =\{(x, v) \mid v \in T(x)\}$, respectively. 
The \emph{inverse} of $T \colon \Ha \tos \Ha$ is $T^{-1} \colon \Ha \tos \Ha$ defined at
any $x\in \Ha$  by $v\in T^{-1}(x)$ if and only if $x\in T(v)$. 
The \emph{sum} of two set-valued maps $T, S\colon \Ha \tos \Ha$  is $T + S \colon \Ha \tos \Ha$, defined by 
the usual Minkowski sum $(T + S)(x) = \set{u + v \mid u\in T(x), v\in S(x)}$. 
For $\lambda > 0$, we also define $\lambda T \colon \Ha \tos \Ha$ by 
$(\lambda T)(x) = \lambda T(x) = \set{ \lambda v \mid v \in T(x)}$.
Whenever necessary, we will also identify single-valued maps $F \colon \Dom F\subset \Ha \to \Ha$ with its set-valued representation $F \colon \Ha\tos \Ha$ by $F(x) = \set{F(x)}$. 
A set-valued map $T \colon \Ha \tos \Ha$ is said to be a \emph{monotone operator} 
if $\inner{x - y}{u - v} \geq 0$ for all $(x, u)$, $(y, v) \in \Gra T$, and \emph{maximal monotone} if it is monotone and its graph $\Gra T$ is not properly contained in the graph of any other monotone operator on $\Ha$.
A single-valued map $F \colon \Dom F\subset \Ha \to \Ha$ is monotone
if $\inner{x - y}{F(x) - F(y)}\geq 0$ for all $x, y \in \Dom F$. The \emph{resolvent} of a maximal monotone operator 
$T \colon \Ha \tos \Ha$ is $J_{T} = (T + I)^{-1}$, where $I$ denotes the identity operator in $\Ha$. 
We will also denote by $\mathcal{S}$ the set of zeroes of the maximal monotone operator $T$, i.e., $\mathcal{S} = T^{-1}(0) = \set{x \mid 0\in T(x)}$.

Let $f \colon \Ha \to \BR$ be an extended real-valued function. The \emph{domain} and \emph{epigraph} of $f$
are $\dom f = \set{x \mid f(x) < +\infty}$ and $\epi f = \set{(x, \mu) \in \Ha\times \R \mid \mu \geq f(x)}$, respectively.
Recall that $f$ is \emph{proper} if $\dom f\neq \emptyset$ and \emph{convex} (resp. \emph{lower semicontinuous}) if $\epi f$ is a convex (resp. closed) subset of
$\Ha\times \R$. 
For $\varepsilon \geq 0$, the \emph{$\varepsilon$-subdifferential} of $f$ is 
$\partial_{\varepsilon}f \colon \Ha \tos \Ha$ defined as 
$\partial_{\varepsilon} f(x) = \set{ v \mid f(y) \geq f(x) + \inner{y-x}{v} - \varepsilon,\;\; \mbox{for all}\; y \in \Ha} $.
When $\varepsilon = 0 $, the maximal monotone~\cite{Roc70b} operator $\partial_{0} f$ is denoted by $\partial f$ and is called the \emph{subdifferential} of $f$. 
For a nonempty closed and convex set $C$, the \emph{indicator function} of $C$ is $\delta_C$ defined by
$\delta_C(x) = 0$ if $x\in C$ and $\delta_C(x) = +\infty$ otherwise. The \emph{normal cone} operator of $C$ is 
the set-valued map $N_C \colon \Ha \tos \Ha$ defined as 
$N_{C}(x) = \{v \mid \inner{v}{y - x} \leq 0,\;\; \mbox{for all}\, y \in C\}$ (we set $N_C(x) = \emptyset$ if $x\notin C$);
it is easy to check that $N_C = \partial \delta_C$.
The \emph{orthogonal projection} onto $C$ is denoted by $P_C$, i.e., $P_C(x)$ is the unique element in $C$ such that $\norm{x - P_C(x)}\leq \norm{x - y}$ for all $y\in C$.

By $\R_+$ and $\R_{++}$ we denote the set of nonnegative and (strictly) positive reals, respectively. 
For $n\in \N$ (the set of natural numbers), we let $[n] \coloneqq \set{0, \dots, n}$.
Finally, strong and weak convergence of sequences will be denoted by the usual symbols $\to$ and $\wto$, respectively.
For more details on notation and basis results on convex analysis and monotone operators we refer the reader 
to~\cite{BC11, Roc70, RW98}.

\section{Preliminaries and basic results} \lab{sec:basic}

In this section, we state some basic and preliminary results, mainly related to the notion of (relative-error) inexact solutions for the proximal subproblems we will employ in this paper. To this end, we begin by recalling some basic facts about $\varepsilon$-enlargements of set-valued maps.

For a set-valued map $T \colon \Ha \tos \Ha$ and $\varepsilon\geq 0$, the $\varepsilon$-enlargement $T^\varepsilon \colon \Ha \tos \Ha$ of $T$ is defined as
\begin{align} \lab{def:enlar}
T^{\varepsilon}(x) = \left\{v \mid \langle u-v ,y-x \rangle \geq - \varepsilon\quad 
\mbox{for all}\;\; (y, u)\in \Gra T\right\}.
\end{align}
If $T$ is monotone, then the inclusion $T(x) \subset T^{\varepsilon}(x)$ trivially holds (this justifies the name ``enlargement''  of $T$).
For $N_C$, i.e., for the normal cone of a (nonempty) closed and convex set, we have
$N_C^\varepsilon = \partial_\varepsilon \delta_C$.
Other relevant and useful properties of $T^\varepsilon$ (for this work) are summarized in the next proposition
(for a proof see, e.g., \cite[Lemma 3.1 and Proposition 3.4(b)]{bur.sag.sva-enl.99}).

\bprop \label{pr:teps}
Let $T, S \colon \Ha\tos \Ha$ be set-valued maps.  Then,
\begin{enumerate}[label = \emph{(\alph*)}]
\item \lab{teps.seg}  If $\varepsilon_1 \leq \varepsilon_2$, then
$T^{\varepsilon_1}(x)\subset T^{\varepsilon_2}(x)$ for all $x \in \Ha$.
\item \lab{teps.ter}  $T^{\varepsilon_1}(x)+S^{\,\varepsilon_2}(x) \subset
(T+S)^{\varepsilon_1+\varepsilon_2}(x)$ for all $x \in \Ha$ and
$\varepsilon_1, \varepsilon_2\geq 0$.
\item \lab{teps.qua} $T$ is monotone if and only if $\Gra T  \subset \Gra T^{0}$.
\item \lab{teps.qui} $T$ is maximal monotone if and only if $T = T^{0}$.
\item \lab{teps.sex} If $f \colon \Ha\to \R\cup \{+\infty\}$ is proper lower semicontinuous and convex, then
  $\partial_\varepsilon f(x)\subset (\partial f)^{\varepsilon}(x)$ for
  all $x\in \Ha$ and $\varepsilon \geq 0$.
\item \lab{teps.sab} If $T$ is maximal monotone, $\{(y_k,v_k,\varepsilon_k)\}$ is
such that $v_k\in T^{\varepsilon_k}(y_k)$, $y_k \wto y$,
$v_k \to v$ and $\varepsilon_k \to \varepsilon$, then
$v\in T^{\varepsilon}(y)$.
\item \lab{teps.dom} If $u \in \partial f (x)$ and $y\in \dom f$, then $u \in \partial_{\varepsilon} f(y)$, where 
$\varepsilon = f(y) - \left(f(x) + \inner{y-x}{u}\right) $.
\end{enumerate}
\eprop

\mgap

\noindent
As we discussed in the introduction, the notion of relative-error criterion for the subproblems
of the exact PP iteration $x_+ = (\lambda T + I)^{-1}x$ can be derived by observing that the iteration can be decoupled as the inclusion-equation system \eqref{eq:dec.prox}, namely,
\begin{align} \label{eq:20}
\begin{cases}
v \in T(y),\\[1mm]
\lambda v + y -x = 0					 
\end{cases}
\end{align}
where $ y = x_+$.
The following concept of approximate solutions for \eqref{eq:20} was introduced for the first time in \cite{SS01}.
\bdefi \label{def:ine-sol}
A triple $(y, v, \varepsilon) \in \Ha \times \Ha \times \R_+$ is an $\sigma$-approximate solution of the proximal system \eqref{eq:20} at $(x, \lambda)\in \Ha\times \R_{++}$ if $\sigma \in [0,1)$ and
\begin{align}\label{def:sigma-solution}
\begin{cases}
v \in T^{\varepsilon}(y),\\[2mm]
\norm{\lambda v + y - x}^2 + 2 \lambda\varepsilon \leq \sigma^2 \left(\norm{\lambda v}^2 + \norm{y-x}^2\right).
\end{cases}
\end{align}
\edefi

Note that if we take $\sigma = 0$ in \eqref{def:sigma-solution}, then it trivially follows that
$\varepsilon = 0$ and $\lambda v + y - x = 0$, which is to say, in view of Proposition \ref{pr:teps}(d),
that $v\in T(y)$ and $\lambda v + y - x = 0$, or in other words, $y = (\lambda T + I)^{-1}x$. Hence the above definition of inexact solution for the subproblems of the PP algorithm gives $y$ as the exact iteration $x_+ = (\lambda T + I)^{-1}x$ in the (exact) case $\sigma = 0$. Moreover, the error in both the inclusion (given by the $\varepsilon$-enlargement) and the equation in \eqref{eq:20} is ``relative'' to the displacement $y - x$ and $\lambda v$.

\mgap

We next state some basic properties of approximate solutions defined above. 
The following lemma was proved in \cite[Lemma 2]{SS01}.

\blemm \lab{lm:equiv}
A triple $ (y, v, \varepsilon)$ is an $\sigma$-approximate solution of the proximal system \eqref{eq:20} at $(x, \lambda)$ \emph{(}in the sense of \emph{Definition \ref{def:ine-sol}}\emph{)}
if and only if 
\begin{align*} 
\langle x - y , v \rangle - \varepsilon \geq \frac{1-\sigma^2}{2\lambda} \left(\norm{ \lambda v}^2 + \norm{x - y}^2\right).
\end{align*}
Furthermore, the following statements are equivalent
\begin{align*}
\begin{cases}
v = 0 , \\
x \in T^{-1}(0),\\
y = x,
\end{cases}
\end{align*}
and any of them implies $\varepsilon = 0$.
\elemm

\mgap

Next lemma will be useful for the analysis (of strong convergence and iteration-complexity) of Algorithm \ref{alg:main} below. Recall that, for a nonempty closed and convex subset $C$ of $\Ha$, the \emph{orthogonal projection} $P_C(x)$ of $x$ onto $C$ is the unique point $P_C(x)$ in $C$ such that 
$\|x - P_C(x)\|\leq \|x - y\|$ for all $y\in C$.

\blemm \lab{lm:olimp}
If $(y, v, \varepsilon)$ is an $\sigma$-approximate solution of the proximal system \eqref{eq:20}
at $(x, \lambda)$, in the sense of \emph{Definition \ref{def:ine-sol}}, then
\begin{align}\label{eq:lower}
\norm{P_H(x) - x} \geq \frac{1 - \sigma^2}{2} \max \left\{ \norm{\lambda v}, \norm{x-y}\right\},
\end{align}
where
\[
H \coloneqq \{z \mid \langle z-y,v \rangle \leq \varepsilon \}.   
\]
\elemm
\begin{proof}
Note first that if $v = 0$, then $H = \Ha$, and so $P_H = I$, which gives that the left-hand side of \eqref{eq:lower} is equal to zero.
On the other hand, in this case, we also have that the right-hand side of \eqref{eq:lower} is zero by Lemma \ref{lm:equiv}.
Let us consider now the case $v\neq 0$. In this case (see, e.g., \cite[Example 28.16]{BC11}), 
\begin{align*}
P_H (x) = x - \frac{\inner{x - y}{v} - \varepsilon}{\norm{v}^2} v
\end{align*}		        
and so
\begin{align}\lab{eq:phx}
\norm{P_H(x) - x} = \frac{\langle x - y,v \rangle - \varepsilon}{\norm{v}}.
\end{align}
Using (again) Lemma \ref{lm:equiv}, we obtain
\begin{align}\lab{eq:phx02}
\begin{aligned}
\frac{ \langle x - y ,v \rangle - \varepsilon}{ \norm{v}} &\geq \frac{1 - \sigma^2}{2} \left(\lambda \norm{v} + \frac{\norm{x - y}^2}{\lambda \norm{v}}\right)\\
 & \geq \frac{1-\sigma^2}{2}\lambda \norm{v}.
\end{aligned}
\end{align}
On the other hand, using the inequality $a + b\geq \sqrt{ab}$ with $a = \lambda\norm{v}$ and 
$b = \normq{x - y}/(\lambda\norm{v})$, we have
\begin{align}\lab{eq:phx03}
\lambda \norm{v} + \frac{\norm{x - y}^2}{\lambda \norm{v}}\geq \norm{x - y}.
\end{align}
The desired result now follows from \eqref{eq:phx}--\eqref{eq:phx03}.
\end{proof}

\mgap

Finally, we will also need the following well-known property of the projection.

\blemm
 Let $C \subset \Ha$ be nonempty, closed and convex and 
let $P_C$ be the orthogonal projection onto $C$. Then, for all $x, y\in \Ha$ and $z\in C$,
\begin{align}
\label{projetion-no-expasive}
 &\norm{P_C(x) - P_C(y) }^2 \leq \norm{x - y}^2 - \norm{\left[x - P_C(x)\right] - \left[y - P_C(y)\right]}^2,\\[4mm]
\label{projetion-caracterization}
&\langle x - P_C(x), z - P_C(x) \rangle \leq 0.
\end{align}
\elemm

\section{A strongly convergent inertial inexact proximal-point algorithm} \lab{sec:main}

In this section, we consider the general (monotone) inclusion problem \eqref{eq:prob}, i.e.,
\begin{align}\lab{eq:probm}
0\in T(x)
\end{align}
where the following assumptions are assumed to hold:
\begin{itemize}
\item[A1.] \lab{ass:a1} 
$T \colon \Ha\tos \Ha$ is a maximal monotone operator.
\item[A2.] \lab{ass:a2} 
The \emph{solution set} $\mathcal{S} \coloneqq T^{-1}(0)$ of \eqref{eq:probm} is nonempty.
\end{itemize}
For numerically solving \eqref{eq:probm}, we present and study the strong convergence and iteration-complexity of a strongly convergent inertial inexact proximal-point (PP) algorithm (see Algorithm \ref{alg:main} below and the comments following it). The main results are Theorems \ref{th:sconv} and \ref{th:pcom}, on strong convergence and (pointwise) iteration-complexity, respectively.

\mgap

Algorithm \ref{alg:main} consists of the following main ingredients:
\begin{description}
\item[Inertial effects] \emph{Inertial effects} on the iteration produced by two ``control'' sequences
$\set{\alpha_k}$ and $\set{\beta_k}$ -- see \eqref{eq:defw}. 
\item[Relative-error criterion for subproblems] A flexible \emph{relative-error criterion} for the proximal subproblems, allowing errors both in the inclusion and in the equation of the proximal system -- see Definition \ref{def:ine-sol} and Eq. \eqref{eq:wood}.   
\item[Orthogonal projection onto half-spaces] Definition of the next iterate $x_{k+1}$ as the \emph{orthogonal projection} onto the intersection of the half-spaces $H_k$ and $W_k$ -- see \eqref{eq:wood02} and \eqref{eq:wood04}.
\end{description}

\mgap

Next comes the algorithm:

\mgap
\mgap

\begin{algorithm}[H] \label{alg:main}
\caption{A strongly convergent inertial inexact PP algorithm for solving  \eqref{eq:probm}}
\SetAlgoLined
\KwInput{$x_0  = x_{-1} \in \Ha$ and $\sigma \in [0,1)$}
\For{$k = 0, 1, 2,\dots$}{
 Choose $\alpha_k, \beta_k \geq 0$ and set
\begin{align} \lab{eq:defw}
\begin{aligned}
w_k &= x_k + \alpha_k (x_k - x_{k-1}),\\[2mm]
\widetilde{w}_k &= w_k + \beta_k (w_k - x_0).
\end{aligned}
\end{align}
\\
 Choose $\lambda_k>0$ and compute an $\sigma$-approximate solution (in the sense of Definition \ref{def:ine-sol}) at $(\widetilde w_k, \lambda_k)$, i.e., compute $(y_k,v_k,\varepsilon_k)$ such that
 \begin{align}
\begin{aligned} \lab{eq:wood}
\begin{cases}
 v_k \in T^{\varepsilon_k}(y_k),\\[2mm]
 \normq{\lambda_k v_k + y_k - \widetilde w_k}
  +2\lambda_k\varepsilon_k
  \leq  \sigma^2\left(\normq{\lambda_k v_k} + \normq{y_k - \widetilde w_k}\right).
\end{cases}
\end{aligned}
 \end{align}  
 \\
 Define
\begin{align} \lab{eq:wood02}
\begin{aligned}
 H_k &= \{z \mid \langle z-y_k,v_k \rangle \leq \varepsilon_k\},\\[2mm]
 W_k &= \{z \mid \langle z-x_k,x_0 - x_k \rangle \leq 0\}.
\end{aligned}
\end{align}
\\
Set
\begin{align}\lab{eq:wood04}
 x_{k+1} = P_{H_k \cap W_k}(x_0).
\end{align}
 }
\end{algorithm}

\mgap
\mgap

Next we make some comments regarding Algorithm \ref{alg:main}.
\begin{enumerate}[label = (\roman*)]
\item The extrapolation steps defined in \eqref{eq:defw} introduce \emph{inertial effects} on the iterations generated by Algorithm \ref{alg:main}. We first mention that $w_k$, defined by the inertial parameter $\alpha_k$, is the typical extrapolation effect usually presented in many inertial-type algorithms (see, e.g., \cite{AA01, AEGM20, AM20, AC20, ACCR18, AP19, BS15}). On the other hand, the additional extrapolation represented by 
$\widetilde w_k$ is specially designed for the particular iteration mechanism of Algorithm \ref{alg:main}.
Conditions on both $\set{\alpha_k}$ and $\set{\beta_k}$ to ensure the (strong) convergence and iteration-complexity analysis of Algorithm \ref{alg:main} will be given in Assumption \ref{ass:ab} (see also Remark \ref{rm:ab} following it).
\item Algorithm \ref{alg:main} is a generalization (an inertial version) of \cite[Algorithm 1]{SS00} and many results of this work are inspired by the latter reference. The relative-error condition \eqref{eq:wood} is more general than the corresponding one  in \cite{SS00}, more precisely \cite[Definition 1]{SS00}; actually \eqref{eq:wood} appeared for the first time in \cite[Definition 1]{SS01} for a different version of the PP algorithm. As we mentioned in the introduction, relative-error criteria similar to \eqref{eq:wood} for subproblems of the PP algorithm and its variants have been employed in different contexts and by many authors (see, e.g.~\cite{AMS16, MS10, MS12, MS13, SS01} and references therein).
\item Note that if we set $\sigma = 0$ in \eqref{eq:wood} we obtain $\varepsilon_k = 0$,
$\lambda_k v_k + y_k - \widetilde w_k = 0$ and $v_k\in T^0(y_k)$, which is to say that
$y_k = (\lambda_k T + I)^{-1}(\widetilde w_k)$ and $v_k = \frac{1}{\lambda_k}(\widetilde w_k - y_k)$ (we recall that, since $T$ is maximal monotone, we have $T^0 = T$). Hence, assuming that the resolvent operator $(\lambda T + I)^{-1}$ is easy to evaluate, it follows that the triple 
$(y_k, v_k, \varepsilon_k) \coloneqq \left((\lambda_k T + I)^{-1}(\widetilde w_k), \frac{1}{\lambda_k}(\widetilde w_k - y_k), 0\right)$ always satisfies the relative-error condition \eqref{eq:wood}. 
In this paper, we don't make such an assumption and are interested in the opposite cases, namely those cases where $(\lambda T + I)^{-1}$ is hard to compute. Prescriptions on how to compute a triple $(y_k, v_k, \varepsilon_k)$ satisfying \eqref{eq:wood} in the general case $\sigma\in [0,1)$ will depend on the particular structure of $T(\cdot)$ in \eqref{eq:probm} -- see, for instance, Sections \ref{sec:korp} and \ref{sec:tseng}.
\item The half-spaces $H_k$ and $W_k$ as in \eqref{eq:wood02} are exactly the same as in \cite[Algorithm 1]{SS00}. They are crucial to ensure the strong convergence of the sequences generated by Algorithm \ref{alg:main}. As it was discussed in \cite{SS00}, the 
definition of $x_{k+1}$ as in \eqref{eq:wood04} depends on the explicit computation of the orthogonal projection onto
$H_k \cap W_k$, which essentially reduces (at most) to the solution of a $2\times 2$ linear system (see \cite[pp. 195 and 196]{SS00} and 
\cite[Propositions 28.18 and 28.19]{BC11}). The well-definedness of Algorithm \ref{alg:main} -- regarding \eqref{eq:wood04} -- will be discussed in 
Appendix \ref{subsec:wd} below. 
\item Algorithm \ref{alg:main} can be used both as a practical method, when an inner algorithm is available to compute a triple $(y_k, v_k, \varepsilon_k)$ satisfying \eqref{eq:wood}, and as a framework for the design and analysis of other more concrete schemes. The latter idea will be further explored in Sections \ref{sec:korp} and \ref{sec:tseng} for designing strongly convergent inertial versions of the Korpelevich, Tseng's modified forward-backward and forward-backward methods for variational inequalities and structured monotone inclusions. 
\end{enumerate}

\subsection{Convergence analysis} \lab{subsec:ca}
In this subsection, we study the (asymptotic) strong convergence of Algorithm \ref{alg:main}; the main result is 
Theorem \ref{th:sconv} below. Propositions \ref{pr:main02} and \ref{pr:main} and Corollary \ref{cor:pacum} are auxiliary results needed to prove the main theorem.  
\bprop \lab{pr:main02}
Consider the sequences evolved by \emph{Algorithm \ref{alg:main}} and let $d_0$ denote the distance of $x_0$
to the solution set $\mathcal{S} \coloneqq T^{-1}(0)\neq \emptyset$ of \eqref{eq:probm}. Then the following statements hold:
\begin{enumerate}[label = \emph{(\alph*)}]
\item \lab{main02.seg} For all $k\geq 0$,
\[
\norm{x_{k+1}-x_0}^2 \geq \norm{x_k - x_0}^2 + \norm{x_{k+1}-x_k}^2.
\]
\item \lab{main02.ter} For all $k\geq 0$,
\[
 \norm{x_k - x_0}\leq d_0.
\]
\item \lab{main02.qua} For all $k\geq 0$,
\begin{align}\lab{eq:inv}
 \norm{x_{k+1} - \widetilde{w}_k}  &\geq \frac{1-\sigma^2}{2} \max \big\{\norm{\lambda_k v_k}, \norm{\widetilde{w}_k - y_k} \big\}.
\end{align}
\end{enumerate}
\eprop
\begin{proof}
\emph{\ref{main02.seg}} By the definition of $W_k$, we have $x_k = P_{W_k}(x_0)$. Applying \eqref{projetion-no-expasive} with $C = W_k$, $x = x_{k+1}$ and $y = x_0$, and taking into account that $x_{k+1} = P_{W_{k}}{(x_{k+1})}$ 
(because $x_{k+1} \in W_k$), we obtain
\begin{align*}
\norm{P_{W_k}(x_{k+1}) - P_{W_k}(x_0)}^2 \leq \norm{x_{k+1}-x_0}^2 - 
\norm{ \underbrace{\left[x_{k+1} - P_{W_k}(x_{k+1})\right]}_{0}  - \left[x_0 - P_{W_k}(x_0)\right]}^2.
\end{align*}
Then
\begin{align*}
\norm{x_{k+1} - x_{k}}^2 \leq \norm{x_{k+1} - x_0}^2 - \norm{x_0 - x_k}^2,
\end{align*}
which is clearly equivalent to the inequality in item \emph{\ref{main02.seg}}.

\mgap

\emph{\ref{main02.ter}} From \eqref{eq:wood04} and \eqref{eq:sshk} below, for all $k\geq 0$,
\begin{align*}
\norm{x_{k+1} - x_0} &= \norm{P_{H_k\cap W_k}(x_0) - x_0}\\
 &\leq \norm{P_{\mathcal{S}}(x_0) - x_0}\\
 &= d_0,
\end{align*}
where $P_{\mathcal{S}}$ denotes the projection onto $\mathcal{S}$. Note that the desired inequality is trivial when $k = 0$.

\mgap

\emph{\ref{main02.qua}} Since $x_{k+1} \in H_k$, using the definition of $P_{H_k}(\widetilde{w}_k)$ we find 
\begin{align*}
\norm{x_{k+1}- \widetilde{w}_k} \geq \norm{P_{H_k}(\widetilde{w}_k) - \widetilde{w}_k}.
\end{align*}
 On the other hand, from Lemma \ref{lm:olimp} and the fact that $(y_k, v_k, \varepsilon_k)$ is an $\sigma$-approximate solution at $(\widetilde w_k, \lambda_k)$ -- see \eqref{eq:wood}--, it follows that
\begin{align*}
\norm{P_{H_k}(\widetilde{w}_k) - \widetilde{w}_k} \geq \frac{1-\sigma^2}{2} \max \{\norm{\lambda_k v_k}, \norm{\widetilde{w}_k - y_k}\},
\end{align*}
which, in turn, finishes the proof of the proposition.
\end{proof}

%\mgap

From now on in this paper we assume that the following conditions hold on the sequences $\set{\alpha_k}$ 
and $\set{\beta_k}$ given in Algorithm \ref{alg:main}.

\bassu \lab{ass:ab}
\begin{enumerate}[label = \emph{(\alph*)}]
\item \lab{ab.seg} The sequence $\set{\alpha_k}$ is bounded, i.e.,
\[
 \overline{\alpha} \coloneqq \sup_{k\geq 0}\, \alpha_k < +\infty.
\]
\item The sequence $\set{\beta_k}$ belongs to $\ell_2$, i.e.,
\[
  \overline s \coloneqq \sum_{k=0}^\infty\,\beta^2_k < +\infty.
\]
We will also use the notation
\[
 \overline{\beta} \coloneqq \sup_{k\geq 0}\, \beta_k < +\infty.
\]
\end{enumerate}
\eassu

\mgap

\begin{remark}\lab{rm:ab}
\emph{
We emphasize that even if we set $\beta_k \equiv 0$, in which case $\overline s = \overline \beta = 0$, the inertial effect produced by the sequence $\set{\alpha_k}$ on Algorithm \ref{alg:main} is still in consonance with the typical extrapolation introduced in the current literature on inertial-type algorithms. More than that,  Assumption \ref{ass:ab}\emph{\ref{ab.seg}}, on the boundedness of the sequence $\set{\alpha_k}$, represents a much more flexible assumption when compared to what is usually imposed --  in the seminal paper \cite{AA01}, for instance, it is assumed that $\overline \alpha < 1/3$. See also \cite[Section 2]{AM20} for an additional discussion. This latter fact makes the results of this paper of interest even in the finite-dimensional setting, where of course weak and strong convergence coincide.
}
\end{remark}

\mgap

\bprop \label{pr:main}
Consider the sequences evolved by \emph{Algorithm \ref{alg:main}}, let $d_0$ denote the distance of $x_0$
to the solution set $\mathcal{S} \coloneqq T^{-1}(0)\neq \emptyset$ of \eqref{eq:probm} and let $\overline \alpha$, $\overline \beta$ and $\overline s$ be as in \emph{Assumption \ref{ass:ab}}. Then
\begin{enumerate}[label = \emph{(\alph*)}]
\item \lab{main.seg}  we have
\[
\sum_{k=0}^\infty\,\norm{x_{k+1}-x_k}^2\leq d_0^2.
\]
\item \lab{main.ter} We have
\[
\sum_{k=0}^\infty\,\norm{x_{k+1} - \widetilde w_k}^2 \leq 
(1 + \overline \alpha)\Big[ (1 + \overline \beta)\big[ 1 + \overline \alpha (1 + \overline \beta) \big] + \overline s \Big] d_0^2.
\]
\item \lab{main.qua} We have
\[
 x_k - y_k \to 0.
\]
\end{enumerate}
\eprop
\begin{proof}
 \emph{\ref{main.seg}} The result follows directly from items \emph{\ref{main02.seg}} and \emph{\ref{main02.ter}} of Proposition \ref{pr:main02} combined with a telescopic sum argument.

\mgap

\emph{\ref{main02.ter}} We claim that
\begin{align}\label{eq:key}
 \norm{w_{k}-z}^2=(1+\alpha_{k})\norm{x_{k}-z}^2
+\alpha_{k}(1+\alpha_{k})\norm{x_{k}-x_{k-1}}^2 -\alpha_{k}\norm{x_{k-1}-z}^2\quad \mbox{for all}\; z\in \Ha.
\end{align}
Indeed, from the definition of $w_k$ as in \eqref{eq:defw}, we have
\[
  x_{k} - z = \frac{1}{1+ \alpha_{k}}(w_{k} - z) + \frac{\alpha_{k}}{1+\alpha_{k}}(x_{k-1} - z)
\]
 and $w_{k}-x_{k-1}=(1+\alpha_{k})(x_{k}-x_{k-1})$, which
combined with  Lemma \ref{lm:imp01}(a) in Appendix \ref{app:aux} and some simple algebra yields \eqref{eq:key}.

 Note now that using  Lemma \ref{lm:imp01}(b) and \eqref{eq:defw}, we have
\begin{align*}
\|x_{k+1} - \tiw_k\|^2& = \|x_{k+1} - w_k - \beta_k(w_k - x_0)\|^2  \notag \\
&= \|x_{k+1} - w_k\|^2 + \beta_k^2\|w_k - x_0\|^2 + 2\beta_k \langle w_k - x_{k+1},  w_k -x_0 \rangle \notag \\
&= (1 + \beta_k)\|w_k - x_{k+1}\|^2 + \beta_k(1+ \beta_k)\|w_k - x_0\|^2 - \beta_k \|x_{k+1} - x_0\|^2.
\end{align*}
Using \eqref{eq:key} (with $z = x_{k+1}$ and $z = x_0$ ) in the above identity, we find
\begin{align}\label{eq:RT}
\|x_{k+1} - \tiw_k\|^2 = &(1 +\beta_k)\left[(1+\alpha_{k})\|x_{k}-x_{k+1}\|^2
+\alpha_{k}(1+\alpha_{k})\|x_{k}-x_{k-1}\|^2 -\alpha_{k}\|x_{k-1}-x_{k+1}\|^2\right]\notag\\
&+\beta_k(1 +\beta_k)\left[(1+\alpha_{k})\|x_{k}-x_0\|^2
+\alpha_{k}(1+\alpha_{k})\|x_{k}-x_{k-1}\|^2 -\alpha_{k}\|x_{k-1}-x_0\|^2\right]\notag\\
& -\beta_k \|x_{k+1} - x_0\|^2\notag\\
=& (1 +\beta_k)(1 +\alpha_k)\Big[\|x_k - x_{k+1}\|^2+ \alpha_k(1+\beta_{k})\|x_{k}-x_{k-1}\|^2
+\beta_{k} \|x_{k}-x_0\|^2 \Big]\notag\\
&- \Big[\alpha_k(1 +\beta_k)\left(\|x_{k-1} - x_{k+1}\|^2+ \beta_{k}\|x_{k-1}-x_0\|^2\right) + \beta_{k} \|x_{k+1}-x_0\|^2  \Big].
\end{align}
On the other hand, since $\|x_{k+1}-x_{k-1}\|+ \|x_{k-1}-x_0\| \geq \|x_{k+1}- x_0\|$, 
using Lemma \ref{lm:min} below, we get
\begin{align}
    \|x_{k+1}-x_{k-1}\|^2+ \beta_{k}\|x_{k-1}-x_0\|^2 \geq \frac{\beta_k}{1 +\beta_k}\|x_{k+1} -x_0\|^2,\notag
\end{align}
which in turn implies
\begin{align}\label{eq:RT0}
    \alpha_k(1 +\beta_k)\left(\|x_{k+1}-x_{k-1}\|^2+ \beta_{k}\|x_{k-1}-x_0\|^2\right) + \beta_{k} \|x_{k+1}-x_0\|^2 \geq \beta_k(1+\alpha_{k}) \|x_{k+1}-x_0\|^2.
\end{align}

Combining \eqref{eq:RT} and \eqref{eq:RT0}, and using the fact that 
$\|x_k - x_0\|\leq \|x_{k+1} -x_0\|$ -- see Proposition \ref{pr:main02}(a) --, we get
\begin{align*}
\|x_{k+1} - \tiw_k\|^2 \leq & (1 +\alpha_k)\bigg[(1 +\beta_k)\big[\|x_{k+1}-x_{k}\|^2+ \alpha_k(1+\beta_{k})\|x_{k}-x_{k-1}\|^2
\big] + \beta_{k}^2 \|x_{k+1}-x_0\|^2\bigg],
\end{align*}
which when combined with item (a) and Proposition \ref{pr:main02}(b), and Assumption \ref{ass:ab}, yields the desired result.

\mgap

\emph{\ref{main02.qua}}  Note first that 
\begin{align*}
  \norm{x_k - y_k} & \leq \norm{x_k - \widetilde{w}_k} + \norm{\widetilde{w}_k - y_k}\\[2mm]
   & \leq \norm{x_{k}- x_{k+1}} + \norm{x_{k+1}- \widetilde{w}_k} + \norm{\widetilde{w}_k - y_k}.
\end{align*}
The desired result now follows from items (a) and (b) and Proposition \ref{pr:main02}(c).
 \end{proof}

\mgap

\bcoro\label{cor:pacum}
Let $\{x_k\}$ be generated by \emph{Algorithm \ref{alg:main}} and assume that there exists $\underline{\lambda} > 0$ such that
\begin{align}\lab{eq:bawfz}
 \lambda_k \geq \underline{\lambda} \quad \mbox{for all}\quad  k\geq 0.
\end{align}
Then the weak cluster points of $\{x_k\}$ belong to $\mathcal{S} := T^{-1}(0)\neq \emptyset$, i.e., they
are solutions of \eqref{eq:probm}.
 \ecoro
 \begin{proof}
Note first that Proposition \ref{pr:main02}(b) implies that $\set{x_k}$ is bounded. From Proposition \ref{pr:main}(b),
\begin{align}\lab{eq:xtwz}
 x_{k+1} - \widetilde w_k \rightarrow 0,
\end{align}
which, in turn, when combined with  \eqref{eq:inv} yields 
 \begin{align*}
  \lambda_k v_k \rightarrow 0 \quad \mbox{and}\quad   \widetilde{w}_k - y_k \rightarrow 0.
 \end{align*}
Consequently, in view of the inequality in \eqref{eq:wood} and
the assumption $\lambda_k\geq \underline{\lambda} > 0$
for all $k\geq 0$, we also have
\begin{align}\lab{eq:etzero}
v_k \to 0\quad \mbox{and}\quad \varepsilon_k \rightarrow 0.    
\end{align}
Let $\{x_{k_j}\}$ be a subsequence of $\set{x_k}$ such that  $x_{k_j} \wto \overline{x}$, where $\overline{x}\in \Ha$. Using Proposition \ref{pr:main}(c) we then obtain 
\begin{align*}
   y_{k_j} \wto \overline{x}.
\end{align*}
On the other hand, from \eqref{eq:etzero}, $v_{k_j} \to 0$ and $\varepsilon_{k_j} \rightarrow 0$.
Altogether, we have $v_{k_j} \in T^{\varepsilon_{k_j}}(y_{k_j})$, $v_{k_j} \to 0$, $y_{k_j} \wto \overline{x}$  and  $\varepsilon_{k_j} \to 0$, which combined with Proposition \ref{pr:teps}\emph{\ref{teps.sab}} yields
\begin{align*}
0 \in T(\overline{x}),\;\;\mbox{i.e.},\;\; \overline x\in \mathcal{S}.  
\end{align*}
Therefore, all weak cluster points of $\{x_k\}$ belong to $\mathcal{S}$, which concludes the proof of the corollary.
\end{proof}
 
Next we will prove the strong convergence of any sequence $\{x_k\}$ generated by Algorithm \ref{alg:main}.

\btheo[Strong convergence of Algorithm \ref{alg:main}] \lab{th:sconv}
Let $\{x_k\}$, $\set{y_k}$ and $\{\lambda_k\}$ be generated by \emph{Algorithm 1} and assume that
$\{\lambda_k\}$ is bounded away from zero, i.e., that condition \eqref{eq:bawfz} holds.
Suppose also that \emph{Assumption \ref{ass:ab}} holds on the sequences $\set{\alpha_k}$ and $\set{\beta_k}$.
Then $\{x_k\}$ and $\set{y_k}$ converge strongly to $x^\ast \coloneqq P_{\mathcal{S}}(x_0)$.
\etheo
\begin{proof}
In view of  Proposition \ref{pr:main02}(b) and the fact that $d_0 = \norm{x^* - x_0}$, we have
\begin{align}\label{eq:xkb}
       \norm{x_k -x_0} \leq \norm{x^{\ast} - x_0} \quad \mbox{for all}\quad k\geq 0.
\end{align}
  Then $\{x_k\}$ is bounded and, by Corollary \ref{cor:pacum}, its weak cluster points belong to $\mathcal{S}$. 
  Let $\{x_{k_j}\}$ be a subsequence of $\{x_k\}$ and $\overline{x} \in \mathcal{S}$ be such that
 \begin{align*}
   x_{k_j} \rightharpoonup \overline{x}.
  \end{align*}
Since $\norm{\cdot}$ is lower semicontinuous in the weak topology, we obtain
\begin{align}\label{eq:ast}
    \norm{\overline{x} - x_0} \leq \liminf \norm{x_{k_j} - x_0} \leq \limsup \norm{x_{k_j} - x_0} \leq \norm{x^{\ast} -x_0},
\end{align}
where in the latter inequality we also used \eqref{eq:xkb}.
 Since $x^{\ast} = P_{\mathcal{S}}(x_0)$ and $\overline{x} \in \mathcal{S}$, as a direct consequence of \eqref{eq:ast}, we obtain $\overline{x} = x^{\ast}$. So, $x_k \wto x^*$, and from \eqref{eq:xkb},
\begin{align*}
\norm{x^* - x_0} \leq \liminf \norm{x_k - x_0} \leq \limsup \norm{x_k - x_0} \leq \norm{x^{\ast} - x_0}
\end{align*}
and so $\norm{x_k - x_0} \to  \norm{x^{\ast} - x_0}$. Since $x_k - x_0 \wto x^* - x_0$, by 
Lemma \ref{lm:imp01}\emph{\ref{imp01:qua}} below, we then obtain
$x_k - x_0 \to x^* - x_0$, which is of course equivalent to $x_k \to x^*$.

The fact that $y_k \to x^*$ follows from the fact that $x_k \to x^*$ and Proposition \ref{pr:main}(c).
\end{proof}

\subsection{Iteration-complexity analysis} \lab{subsec:ic}
In this subsection, we study the (pointwise) iteration-complexity of Algorithm \ref{alg:main}.
The main result is Theorem \ref{th:pcom} below.
Motivated by the complexity analyzes of the HPE method presented in \cite{MS10}, we study the iteration-complexity of Algorithm \ref{alg:main} in the following sense: for a given tolerance $\rho > 0$, we wish to estimate the number of iterations needed to compute a triple $(y, v, \varepsilon)\in \Ha\times \Ha\times \R_+$ such that 
\begin{align} \lab{eq:apros}
v\in T^\varepsilon(y)\;\;\mbox{and}\;\;\max\{\norm{v}, \varepsilon\}\leq \rho.
\end{align}
In this case, $y$ is considered to be a $\rho$-approximate solution of \eqref{eq:probm}. Note that if we set 
$\rho = 0$ in \eqref{eq:apros}, then $v = 0$ and $\varepsilon = 0$, in which case $0 \in T^0(y)$, i.e, 
$0 \in T(y)$ (in other words $y$ is a solution of \eqref{eq:probm}).

\mgap

Next is the main theorem regarding the iteration-complexity of Algorithm \ref{alg:main}.

\btheo[Pointwise iteration-complexity of Algorithm \ref{alg:main}] \lab{th:pcom}
Consider the sequences evolved by \emph{Algorithm 1} and assume that $\set{\lambda_k}$ is bounded away from zero, i.e., assume that condition \eqref{eq:bawfz} holds. Let $d_0$ denote the distance of $x_0$ to the solution set $\mathcal{S} \coloneqq T^{-1}(0)\neq \emptyset$ of \eqref{eq:probm} and suppose
\emph{Assumption \ref{ass:ab}} holds. Then, for every $k \geq 0$, there exists $j \in [k]$ such that
\begin{subnumcases}{ }
 v_j \in T^{\varepsilon_j}(y_j), & \lab{pervj} \\[3mm]
\norm{v_j} \leq \dfrac{2d_0}{\sqrt{k+1}}\left( \frac{\sqrt{(1 + \overline \alpha)\Big[(1 + \overline \beta)
\big[ 1 + \overline \alpha (1 + \overline \beta) \big] + \overline s\Big]}}{(1- \sigma ^2)\underline{\lambda}}\right), & \label{des:v_jcom}
\\[3mm]
\varepsilon_j \leq \frac{d_0^2}{k+1}\left(\frac{4 \sigma^2 (1 + \overline \alpha)\Big[(1 + \overline \beta)
\big[ 1 + \overline \alpha (1 + \overline \beta) \big] + \overline s\Big]}{ (1 -\sigma^2)^2 \underline{\lambda}}\right). & \label{des:e_jcom}
\end{subnumcases}
\etheo
\begin{proof}
Using Proposition \ref{pr:main02}(c) and the inequality in \eqref{eq:wood}, we find for all $i \in [k]$,
\begin{align*}
\max\left\{\normq{\lambda_i v_i}, 
\frac{\lambda_i \varepsilon_i}{\sigma^2} \right\}\leq 
\dfrac{4}{(1-\sigma^2)^2}\normq{x_{i+1} - \widetilde w_i}.
\end{align*}
Hence, from Proposition \ref{pr:main}(b),
\begin{align*}
\sum_{i=0}^{k}\,\max\left\{\normq{\lambda_i v_i}, 
\frac{\lambda_i \varepsilon_i}{\sigma^2} \right\}\leq
\dfrac{4(1 + \overline \alpha)\Big[(1 + \overline \beta)\big[ 1 + \overline \alpha (1 + \overline \beta) \big] + 
\overline s\Big] d_0^2}{(1-\sigma^2)^2}
\end{align*}
and so there exists $j\in [k]$ such that
\begin{align*}
(k+1)\,\left(\max\left\{\normq{\lambda_j v_j}, 
\frac{\lambda_j \varepsilon_j}{\sigma^2} \right\}\right)\leq
\dfrac{4(1 + \overline \alpha)\Big[(1 + \overline \beta)\big[ 1 + \overline \alpha (1 + \overline \beta) \big] + 
\overline s\Big] d_0^2}{(1-\sigma^2)^2},
\end{align*}
which combined with \eqref{eq:bawfz} yields \eqref{des:v_jcom} and \eqref{des:e_jcom}. To finish the proof of the theorem, note that
\eqref{pervj} follows directly from the inclusion in \eqref{eq:wood}.
\end{proof}

We now make some remarks regarding Theorem \ref{th:pcom}.
\begin{enumerate}[label = (\roman*)]
\item If we set $\beta_k\equiv 0$ in Algorithm \ref{alg:main} -- see Remark \ref{rm:ab} -- then it follows that 
$\overline \beta = \overline s = 0$ (see Assumption \ref{ass:ab}) and consequently \eqref{pervj} -- \eqref{des:e_jcom}
reduce to
\begin{align}
 \lab{eq:veloso}
 v_j\in T^{\varepsilon_j}(y_j),\quad \norm{v_j}\leq \dfrac{2(1 + \overline \alpha)d_0}{\sqrt{k + 1}\,(1-\sigma^2)\underline \lambda},
\quad \varepsilon_j\leq \dfrac{4\sigma^2(1 + \overline \alpha)^2 d_0^2}{(k + 1)\,(1-\sigma^2)^2\underline \lambda}.
\end{align}
Here we emphasize that the bounds on $v_j$ and $\varepsilon_j$ given in \eqref{eq:veloso} are much simpler than the corresponding ones obtained in \cite[Theorem 2.7]{AM20}, in a previous attempt to study the iteration-complexity of relative-error inexact PP methods with inertial effects. This improvement was achieved essentially at the cost of computing an extra projection onto the intersection of two half-spaces (see the definition of $x_{k+1}$ as in \eqref{eq:wood04}), which, as mentioned in the fourth remark
following Algorithm \ref{alg:main}, reduces to solving (at most) a $2\times 2$ linear system of equations.
Furthermore, as we also pointed out in Remark \ref{rm:ab}, the assumption that $\set{\alpha_k}$ is bounded 
(see Assumption \ref{ass:ab}) is considerably weaker than what is usually required in the analysis of inertial-type PP algorithms (in particular in~\cite{AM20}).  
\item For simplicity, let us consider \eqref{eq:veloso}. In this case, Theorem \ref{th:pcom} guarantees that, for a given tolerance $\rho > 0$, Algorithm \ref{alg:main} finds a triple
$(y, v, \varepsilon)$ satisfying \eqref{eq:apros} in at most
\begin{align}
O\left(\left\lceil\left(\dfrac{1 + \overline \alpha}{1- \sigma^2}\right)^2\max\left\{\left(\dfrac{d_0}{\underline \lambda \rho}\right)^2, 
\dfrac{\sigma^2  d_0^2}{\underline \lambda \rho}\right\}\right\rceil\right)
\end{align}
iterations.
\end{enumerate}

%\newpage
\section{A strongly convergent inertial variant of the Korpolevich extragradient method}
 \lab{sec:korp}
In this section, we consider the VIP \eqref{eq:prob04}, i.e., the problem of finding $x\in C$ such that  
\begin{align} \lab{eq:prob04m}
\inner{F(x)}{y - x}\geq 0 \quad \mbox{for all} \quad y\in C,
\end{align}
where the following assumptions are assumed to hold:
\begin{itemize}
\item[B1.] $C \subset \Dom F$ is a nonempty closed and convex subset of $\Ha$.
\item[B2.] $F \colon \Dom F\subset \Ha \to \Ha$ is monotone and $L$-Lipschitz continuous, i.e., $F$ is monotone (on $\Dom F$) and there exists $L > 0$ such that
 \begin{align} \label{eq:f.Lip}
  \norm{F(x) - F(y)}\leq L\norm{x - y} \quad \mbox{for all} \quad x,y\in \Dom F.
 \end{align}
\item[B3.] The solution set $\mathcal{S} \coloneqq (F + N_C)^{-1}(0)$ of \eqref{eq:prob04m} is nonempty.
\end{itemize}
Note that under the above assumptions on $C$ and $F(\cdot)$ the operator $F + N_C$ is maximal monotone~\cite[Proposition 12.3.6]{FPII03}. 
Moreover, as we discussed in the introduction, problem \eqref{eq:prob04m} is equivalent to the monotone inclusion $0\in T(x)$ with
$T = F + N_C$; this fact allow us to apply the results of Section \ref{sec:main} to \eqref{eq:prob04m}. 

\mgap

For numerically solving \eqref{eq:prob04m}, we propose and study a strongly convergent inertial variant of the celebrated Korpolevich extragradient method~\cite{Kor76} (see Algorithm \ref{alg:korp}). The main results regarding strong convergence and iteration-complexity are summarized in Theorem \ref{th:korp} below.

\mgap
\mgap

\begin{algorithm}[H] \label{alg:korp}
\caption{A strongly convergent inertial variant of the Korpolevich extragradient method for solving \eqref{eq:prob04m}}
\SetAlgoLined
\KwInput{$x_0 = x_{-1} \in \Ha$, $\sigma \in (0,1)$ and $\lambda = \sigma/L$}
\For{$k = 0, 1, 2, \dots$}{
  Choose $\alpha_k, \beta_k \geq 0$ and set
\begin{align} \lab{eq:defw_korp}
\begin{aligned}
w_k &= x_k + \alpha_k (x_k - x_{k-1}),\\[2mm]
\widetilde{w}_k &= w_k + \beta_k (w_k - x_0).
\end{aligned}
\end{align}
\\
Let $w'_k = P_C(\widetilde w_k)$  and compute 
\begin{align} \lab{eq:yty}
\begin{aligned}
&y_k = P_{C}\left(\tiw_k - \lambda F(w'_k)\right),\\[2mm]
&\widetilde y_k = P_{C} \left(\tiw_k -\lambda F(y_k)\right).
\end{aligned}
\end{align}
\\
Define
\begin{align} \lab{eq:HW_korp}
\begin{aligned}
 H_k &= \{z \mid \langle z-y_k,v_k \rangle \leq \varepsilon_k\},\\[2mm]
 W_k &= \{z \mid \langle z-x_k,x_0 - x_k \rangle \leq 0\},
\end{aligned}
\end{align}
where 
\begin{align}
  \lab{eq:qve_korp}
\begin{aligned}
&q_k = \frac{\tiw_k - \widetilde y_k}{\lambda} - F(y_k),\\[1mm]
&v_k = F(y_k) + q_k,\\[1mm]
&\varepsilon_k = \inner{q_k}{\widetilde y_k - y_k}.
\end{aligned}
\end{align}
\\
Set
\begin{align}\lab{eq:P_korp}
 x_{k+1} = P_{H_k \cap W_k}(x_0).
\end{align}
  }
\end{algorithm}

\mgap
\mgap

\noindent
Next we make some remarks regarding Algorithm \ref{alg:korp}.
\begin{enumerate}[label = (\roman*)]
\item Here the role of the (inertial) parameters $\alpha_k$ and $\beta_k$ as in \eqref{eq:defw_korp} 
is the same as in the algorithm of the previous section -- see  the first remark following Algorithm \ref{alg:main}. Analogously to the previous section, to study the strong convergence and iteration-complexity of Algorithm \ref{alg:korp}, 
we will also assume the same conditions on both $\set{\alpha_k}$ and $\set{\beta_k}$ as in Assumption \ref{ass:ab}.
\item The crucial step in Algorithm \ref{alg:korp} is the computation of the projections as in \eqref{eq:yty}. This is the same mechanism of the famous Korpolevich's extragradient method: $y_k$ is a ``projected gradient'' step at $\widetilde w_k$ in the direction of $-F(w'_k)$, while $\widetilde y_k$ is a correction (extragradient) step. 
The extra projection to compute $w'_k = P_C(\widetilde w_k)$ is needed to recover feasibility for $\widetilde w_k$ with respect to $C$ (otherwise it could be impossible to evaluate $F(\cdot)$ in the case where $\widetilde w_k$ doesn't belong to $\Dom F$). We mention that if $F(\cdot)$ is defined in the whole space $\Ha$, then one can simply take $w'_k = \widetilde w_k$ and no extra projection would be required in this case. 
\item The half-spaces $H_k$ and $W_k$ as in \eqref{eq:HW_korp} are defined in the same way as in Algorithm \ref{alg:main} --
see \eqref{eq:wood02}. Note also that the computation of $q_k$, $v_k$ and $\varepsilon_k$ as in \eqref{eq:qve_korp} doesn't require extra evaluations of the operator $F(\cdot)$. Also, the computation of $x_{k+1}$ given in \eqref{eq:P_korp} was discussed in the fourth remark following Algorithm \ref{alg:main}.
\item Since Algorithm \ref{alg:korp} is a special instance of 
Algorithm \ref{alg:main} (see Proposition \ref{pr:korp_sp} below), it follows from Corollary \ref{cor:2} in Appendix \ref{subsec:wd} that Algorithm \ref{alg:korp} is also well-defined; see also the fourth remark following Algorithm \ref{alg:main} and Proposition \ref{pr:korp_sp} below.
\end{enumerate}

\mgap

The proof of the next proposition follows the same outline of \cite[Theorem 5.1]{MS10}. For the sake of completeness, we include a proof here.

\bprop \lab{pr:korp_sp}
For the sequences evolved by \emph{Algorithm \ref{alg:korp}}, the following holds for all $k\geq 0$:
 \begin{align} \label{eq:kp01}
   \begin{aligned}
   \begin{cases}
    v_k\in (F+N_C^{\,\varepsilon_k})(y_k) \subset (F + N_C)^{\varepsilon_k}(y_k),\\[2mm]
    \norm{\lambda v_k+y_k - \tiw_{k}}^2 + 2\lambda \varepsilon_k\leq \sigma^2 \norm{y_k-\tiw_{k}}^2.
\end{cases}
 \end{aligned}
\end{align}
As a consequence, it follows that \emph{Algorithm \ref{alg:korp}} is a special instance of 
\emph{Algorithm \ref{alg:main}} \emph{(}with $\lambda_k\equiv \lambda$\emph{)} for solving \eqref{eq:probm} with 
$T = F + N_{C}$.
\eprop
\begin{proof}
Direct use of the second identity in \eqref{eq:yty} and the fact that $P_C = (\lambda N_C + I)^{-1}$ give
$q_k \in N_C(\widetilde y_k) = \partial \delta_C(\widetilde y_k)$, where $q_k$ is as in \eqref{eq:qve_korp}, and so by Proposition \ref{pr:teps}(g) and the fact that $\partial_{\varepsilon_k}\delta_C = N_C^{\,\varepsilon_k}$, we find
\begin{align*}
 q_k \in \partial_{\varepsilon_k} \delta_C(y_k) = N_C^{\,\varepsilon_k}(y_k)
\end{align*}
where $\varepsilon_k$ is as in \eqref{eq:qve_korp}. From the latter inclusion and the definition of $v_k$ as in 
\eqref{eq:qve_korp} we obtain the first inclusion in \eqref{eq:kp01}. The second inclusion follows from the first one and
Proposition \ref{pr:teps}(b). Let's now prove the inequality in \eqref{eq:kp01}.

Note first that in view of the first identity in \eqref{eq:yty}, 
\begin{align}\lab{eq:kp01x}
 p_k \coloneqq \dfrac{\widetilde w_k - y_k}{\lambda} - F(w'_k)\in N_C(y_k),
\end{align}
which in turn combined with the fact that $\widetilde y_k\in C$ and the definition of $N_C(y_k)$
yields $\inner{p_k}{\widetilde y_k - y_k}\leq 0$. Thus,
\begin{align}\label{eq:e_k}
\varepsilon_k  =  \inner{q_k}{\widetilde y_k - y_k} =  \inner{q_k-p_k}{\widetilde y_k-y_k} + 
\inner{p_k}{\widetilde y_k-y_k} \leq \inner{q_k-p_k}{\widetilde y_k - y_k}
\end{align}
and so
\begin{align*}
 \norm{\lambda v_k + y_k - \tiw_k}^2 + 2 \lambda   \varepsilon_k  
 &= \norm{y_k - \widetilde y_k}^2 + 2 \lambda \varepsilon_k  && \mbox{[\,by \eqref{eq:qve_korp}\,]}\\
 &\leq \norm{y_k - \widetilde y_k}^2 + 2 \lambda \inner{q_k - p_k}{\widetilde y_k - y_k} && \mbox{[\,by \eqref{eq:e_k}\,]}\\
 &= \norm{y_k - \widetilde y_k - \lambda(q_k - p_k)}^2 - \norm{\lambda(q_k-p_k)}^2 \\
 &\leq \norm{y_k - \widetilde y_k - \lambda (q_k - p_k)}^2 \\
 &= \norm{\lambda\left( F(y_k) - F(w'_k)\right)}^2 && \mbox{[\,by \eqref{eq:qve_korp} and \eqref{eq:kp01x}\,]}\\
 &\leq (\lambda L)^2 \norm{y_k - w'_k}^2 &&  \mbox{[\,by \eqref{eq:f.Lip}\,]}\\
 &\leq (\lambda L)^2 \norm{y_k - \widetilde w_k}^2 \\
 &= \sigma^2 \norm{y_k - \widetilde w_k}^2 && \mbox{[\,by the fact that $\lambda = \sigma/L$\,].}
\end{align*}

Finally, the last statement of the proposition follows from \eqref{eq:kp01} and the definitions of Algorithms \ref{alg:main}
and \ref{alg:korp}.
\end{proof}

Next we summarize our main findings on Algorithm \ref{alg:korp}.

\btheo[{Strong convergence and iteration-complexity of Algorithm \ref{alg:korp}}] \lab{th:korp}
Suppose assumptions \emph{B1}, \emph{B2} and \emph{B3} as above hold on $C$ and $F(\cdot)$ and consider the sequences evolved by \emph{Algorithm \ref{alg:korp}}. Let $d_0$ denote the distance of $x_0$ to the solution set  $\mathcal{S}\neq \emptyset$ of \eqref{eq:prob04m} and suppose \emph{Assumption \ref{ass:ab}} holds on the sequences $\set{\alpha_k}$ and $\set{\beta_k}$. 
Then the following statements hold:
\begin{enumerate}[label = \emph{(\alph*)}]
\item The sequences $\{x_k\}$ and $\set{y_k}$ converge strongly to $x^\ast \coloneqq P_{\mathcal S}(x_0)$.
\item  For all $k\geq 0$, there exists $j\in [k]$ such that
\begin{subnumcases}{ }
v_j\in (F + N_C^{\,\varepsilon_j})(y_j), & \label{pervj02}\\[3mm]
\norm{v_j} \leq \dfrac{2 d_0 L}{\sqrt{k+1}}\left( \frac{\sqrt{(1 + \overline \alpha)\left[(1 + \overline \beta)\left( 1 + \overline \alpha (1 + \overline \beta) \right) + \overline s\right]}}{(1- \sigma ^2)\sigma}\right), & \label{des:vjc2} \\[3mm]
\varepsilon_j \leq \frac{d_0^2 L}{k+1}\left(\frac{4 \sigma (1 + \overline \alpha)\left[(1 + \overline \beta)\left( 1 + \overline \alpha (1 + \overline \beta) \right) + \overline s\right]}{ (1 -\sigma^2)^2}\right).& \label{des:ejc2} 
\end{subnumcases}
\end{enumerate}
\etheo
\begin{proof}
By Proposition \ref{pr:korp_sp}, we know that Algorithm \ref{alg:korp} is a special instance of Algorithm \ref{alg:main} for solving
\eqref{eq:probm} with $T = F + N_C$. Hence, we just have to call Theorems \ref{th:sconv} and \ref{th:pcom}, use the fact that
$\lambda_k\equiv \lambda = \sigma/L$ (see the input in Algorithm \ref{alg:korp}) and recall the VIP
\eqref{eq:prob04m} is equivalent to the inclusion $0\in (T + N_C)(x)$.
\end{proof}

\mgap

\begin{remark} \lab{rm:chris}
\emph{
The bounds in item (b) above for $\norm{v_j}$ and $\varepsilon_j$ are similar to the corresponding ones found by Monteiro and Svaiter in~\cite[Theorem 5.2(a)]{MS10} for the Korpolevich extragradient method (see also \cite{Nem05}). Similarly to the first remark following Theorem \ref{th:pcom}, if we set $\beta_k\equiv 0$ in Algorithm \ref{alg:korp}, then 
\eqref{pervj02} -- \eqref{des:ejc2} reduce to
\begin{align*}
 v_j\in (F+N_C^{\varepsilon_j})(y_j),\quad \norm{v_j}\leq \dfrac{2(1 + \overline \alpha)d_0 L}{\sqrt{k + 1}\,(1-\sigma^2)\sigma},
\quad \varepsilon_j\leq \dfrac{4\sigma (1 + \overline \alpha)^2 d_0^2 L}{(k + 1)\,(1-\sigma^2)^2}.
\end{align*}
The latter conditions imply that for a given tolerance $\rho > 0$, Algorithm \ref{alg:korp} finds a triple
$(y, v, \varepsilon)$ such that $v\in (F + N_C^\varepsilon)(y)$ and $\max\{\norm{v}, \varepsilon\}\leq \rho$ in at most
\begin{align}
O\left(\left\lceil\left(\dfrac{1 + \overline \alpha}{1- \sigma^2}\right)^2
\max\left\{\left(\dfrac{d_0 L}{\sigma\rho}\right)^2, 
\dfrac{d_0^2 L}{\rho}\right\}\right\rceil\right)
\end{align}
iterations. Moreover, we also mention that the inclusion \eqref{pervj02} is closely related to the more usual notions of \emph{weak and strong solutions} of VIPs (see, e.g., \cite{MS10} for a discussion).
}
\end{remark}

\section{Strongly convergent inertial variants of Tseng's modified forward-backward and forward-backward methods} \lab{sec:tseng}
In this section, we consider the structured monotone inclusion problem \eqref{eq:prob02}, i.e., we consider the problem of finding $x\in \Ha$ such that
\begin{align} \lab{eq:mipt}
 0\in F(x) + B(x)
\end{align}
where the following conditions are assumed to hold:
\begin{itemize}
\item[C1.] $F \colon \Dom F\subset \Ha\to \Ha$ is a (single-valued) continuous monotone map.
\item[C2.] $B \colon \Ha\tos \Ha$ is a (set-valued) maximal monotone operator such that $\Dom B \subset C \subset \Dom F$,
where $C$ is a nonempty closed and convex subset of $\Ha$.
\item[C3.] The solution set $\mathcal{S} \coloneqq (F + B)^{-1}(0)$ of \eqref{eq:mipt} is nonempty.
\end{itemize} 
Conditions C1 and C2 above guarantee that the operator $T = F + B$ is maximal monotone 
(see~\cite[Proposition A.1]{MS11}), and so we can apply the results of Section \ref{sec:main} for solving \eqref{eq:mipt}.

Assuming that the resolvent $(\lambda B + I)^{-1}$ of $B(\cdot)$ is easy to evaluate, we propose two forward-backward type methods for numerically solving \eqref{eq:mipt}:
\begin{itemize}
\item A strongly convergent inertial variant of the Tseng's forward-backward method~\cite{Tse20} for the case that 
$F(\cdot)$ is $L$-Lipschitz continuous, that is, in the case where \eqref{eq:f.Lip} holds. The method is presented as 
Algorithm \ref{alg:tseng} below, and the main results on strong convergence and iteration-complexity are summarized in Theorem \ref{th:tseng}.
\item A modification of the forward-backward method for the case where 
$F(\cdot)$ is $(1/L)$-cocoercive on $\Ha$:
 \begin{align}  \label{eq:f.coco}
   \inner{x - y}{F(x) - F(y)}\geq \dfrac{1}{L}\norm{F(x) - F(y)}^2\quad \mbox{for all}\quad  x, y\in \Ha,
  \end{align}
 for some $L > 0$. The proposed algorithm appears as Algorithm \ref{alg:fb} below and the main results on convergence and complexity are summarized in Theorem \ref{th:fb.main}.
 \end{itemize}

\subsection{A strongly convergent inertial variant of the Tseng's forward-backward method} \label{subsec:tsg}

As we mentioned in the beginning of this section, in this subsection we consider problem \eqref{eq:mipt}
under the assumption that $F(\cdot)$, $B(\cdot)$ and/or $C$ satisfy assumptions C1, C2 and C3 and the $L$-Lipschitz conditions
\eqref{eq:f.Lip}.

\mgap

Next is the algorithm:

\mgap

\begin{algorithm}[H] \lab{alg:tseng}
\caption{A strongly convergent inertial variant of the Tseng's forward-backward method for solving \eqref{eq:mipt}}
\SetAlgoLined
\KwInput{$x_0 = x_{-1} \in \Ha$, $\sigma \in (0,1)$ and $\lambda = \sigma/L$}
\For{$k = 0, 1, 2, \dots$}{
Choose $\alpha_k, \beta_k \geq 0$ and set
\begin{align} \lab{eq:defw_Ts}
\begin{aligned}
w_k &= x_k + \alpha_k (x_k - x_{k-1}),\\[2mm]
\widetilde{w}_k &= w_k + \beta_k (w_k - x_0).
\end{aligned}
\end{align}
\\
Let $w'_{k} = P_C(\tiw_{k})$ and compute
\begin{align} \lab{eq:itse02}
\begin{aligned} 
 &y_k = (\lambda B+I)^{-1}\left(\tiw_{k} - \lambda F(w'_{k})\right),\\[2mm]
 & v_k = F(y_k) - F(w'_k) + \dfrac{1}{\lambda}\left(\tiw_{k} - y_k\right).
\end{aligned}
\end{align}
\\
Define
\begin{align} \lab{eq:HW_Ts}
\begin{aligned}
 H_k &= \{z \mid \langle z-y_k,v_k \rangle \leq 0\},\\[2mm]
 W_k &= \{z \mid \langle z-x_k,x_0 - x_k \rangle \leq 0\}.
\end{aligned}
\end{align}
\\
Set
\begin{align}\lab{eq:P_Ts}
 x_{k+1} = P_{H_k \cap W_k}(x_0).
\end{align}
}
\end{algorithm}

\mgap
\mgap

\noindent
We now make the following remarks regarding Algorithm \ref{alg:tseng}:
\begin{enumerate}[label = (\roman*)]
\item We refer the reader to the several comments following both Algorithms \ref{alg:main} and \ref{alg:korp} regarding
the role of the (inertial) sequences $\set{\alpha_k}$ and $\set{\beta_k}$ as well as of the half-spaces $H_k$ and $W_k$, 
and also regarding the well-definedness of Algorithm \ref{alg:tseng} with respect to the projection in \eqref{eq:P_Ts}.
\item From a numerical point of view, the most expensive operation in Algorithm \ref{alg:tseng} is the computation of the resolvent $(\lambda B + I)^{-1}$ of $B$ in \eqref{eq:itse02}; the computation of $y_k$ and $v_k$ as in \eqref{eq:itse02} resembles the  iteration of the Tseng's modified forward-backward method~\cite{Tse20}. We also mention that Algorithm \ref{alg:tseng}
can also be applied to solve the VIP \eqref{eq:prob04m} by taking $B = N_C$ in \eqref{eq:mipt}, in which case
$(\lambda B + I)^{-1} = P_C$. In this case, comparing Algorithms \ref{alg:korp} and \ref{alg:tseng}, it appears that the main advantage of the latter over the former is that \eqref{eq:itse02} requires the computation of one projection, namely
$y_k = P_C\left(\tiw_{k} - \lambda F(w'_{k})\right)$, while \eqref{eq:yty} requires two of them. Here we also mention that
the extra projection $w'_{k} = P_C(\tiw_{k})$ in Step 2 of Algorithm \ref{alg:tseng} is needed to recover feasibility of
$\widetilde w_k$ with respect to the domain $\Dom F$ of $F(\cdot)$; if $\Dom F = \Ha$, the one can simply take
$w'_{k} = \tiw_{k}$.
\item Other potential advantage of Algorithm \ref{alg:tseng} over Algorithm \ref{alg:korp} (when both are applied to solve
\eqref{eq:prob04m}) is the quality of the approximation: while for the former one obtains $(y_k, v_k)$ exactly in the graph of 
$F + N_C$, for the latter $(y_k, v_k)$ belongs to an enlargement of $F + N_C$ -- see the inclusions in \eqref{eq:kp01}
and \eqref{eq:ts02}. 
\end{enumerate}

\bprop \label{pr:temain}
For the sequences evolved by \emph{Algorithm \ref{alg:tseng}}, the following holds for all $k\geq 0$:
 \begin{align}\label{eq:ts02}
   \begin{aligned}
  \begin{cases}
   v_k\in (F + B)(y_k),\\[2mm]
   \norm{\lambda v_k+y_k-\tiw_{k}}\leq \sigma \norm{y_k-\tiw_{k}}.
\end{cases}
 \end{aligned}
\end{align}
As a consequence, by letting $\lambda_k\equiv \lambda$ and $\varepsilon_k \equiv 0$, it follows that \emph{Algorithm \ref{alg:tseng}} is a special instance of \emph{Algorithm \ref{alg:main}} for solving \eqref{eq:probm}  with $T = F+B$.
 \eprop
\begin{proof}
The proof of \eqref{eq:ts02} follows the same outline of the proof of \cite[Proposition 6.1]{MS10}. The fact that 
Algorithm \ref{alg:tseng} is a special instance of Algorithm \ref{alg:main} follows from \eqref{eq:ts02}, 
Algorithms \ref{alg:tseng} and \ref{alg:main}'s definitions and Proposition \ref{pr:teps}(d).
\end{proof}

Next is the main result on Algorithm \ref{alg:tseng}:

\begin{theorem}[Strong convergence and iteration-complexity of Algorithm \ref{alg:tseng}] \lab{th:tseng}
Suppose assumptions \emph{C1}, \emph{C2} and \emph{C3} and the $L$-Lipschitz condition \eqref{eq:f.Lip} as above hold. Consider the sequences evolved by \emph{Algorithm \ref{alg:tseng}} and let $d_0$ denote the distance of $x_0$ to the solution set  $\mathcal{S}\neq \emptyset$ of \eqref{eq:mipt}, and suppose \emph{Assumption \ref{ass:ab}} holds on the sequences $\set{\alpha_k}$ and $\set{\beta_k}$.  Then the following statements hold:
\begin{enumerate}[label = \emph{(\alph*)}]
\item The sequences $\{x_k\}$ and $\set{y_k}$ converge strongly to $x^\ast \coloneqq P_S(x_0)$.
\item For all $k\geq 0$, there exists $j\in [k]$ such that
\begin{subnumcases}{ }
v_j\in (F + B)(y_j), & \label{pervj03}\\[3mm]
\norm{v_j} \leq \dfrac{2 d_0 L}{\sqrt{k+1}}\left( \frac{\sqrt{(1 + \overline \alpha)\left[(1 + \overline \beta)\left( 1 + \overline \alpha (1 + \overline \beta) \right) + \overline s\right]}}{(1- \sigma ^2)\sigma}\right). & \label{des:vc03}
\end{subnumcases}
\end{enumerate}
\end{theorem}
\begin{proof}
The proof follows the same outline of Theorem \ref{th:korp}'s proof, now using Proposition \ref{pr:temain} instead of Proposition \ref{pr:korp_sp}.
\end{proof}

\mgap

\brema \lab{rm:vega}
\emph{
Similarly to Remark \ref{rm:chris}, it is possible to use \eqref{pervj03} and \eqref{des:vc03} in order to estimate the number of iterations needed by Algorithm \ref{alg:tseng} to produce a pair $(y, v)$ such that
\begin{align} \lab{eq:vega02}
 v\in (F + B)(y)\;\;\mbox{and}\;\;\norm{v}\leq \rho
\end{align}
where $\rho > 0$ is a given tolerance. As we mentioned in the third remark following Algorithm \ref{alg:tseng}, one apparent advantage of
\eqref{eq:vega02} over the corresponding result in Remark \ref{rm:chris} is that the former gives a pair exactly in the graph of $F+B$ while the latter relies on enlargements.
}
\erema

\subsection{A strongly convergent inertial version of the forward-backward method} \label{subsec:fb}
In this subsection, we consider problem \eqref{eq:mipt}
under the assumption that $F(\cdot)$, $B(\cdot)$ and/or $C$ satisfy assumptions C1, C2 and C3 and the 
(cocoercivity) condition \eqref{eq:f.coco}.

\mgap
\mgap

\begin{algorithm}[H] \label{alg:fb}
\caption{A strongly convergent inertial version of the forward-backward method for solving \eqref{eq:mipt}}
\SetAlgoLined
\KwInput{$x_0 = x_{-1} \in \Ha$, $\sigma \in (0,1)$ and $\lambda = 2\sigma^2/L$}
\For{$k = 0, 1, 2, \dots$}{
Choose $\alpha_k, \beta_k \geq 0$ and set
\begin{align} \lab{eq:defw_fb}
\begin{aligned}
w_k &= x_k + \alpha_k (x_k - x_{k-1}),\\[2mm]
\widetilde{w}_k &= w_k + \beta_k (w_k - x_0).
\end{aligned}
\end{align}
\\
Let $w'_k = P_C(\widetilde w_k)$ and compute
\begin{align} \label{eq:fb}
 y_k = (\lambda B + I)^{-1}\left(\tiw_{k}-\lambda F(w'_{k})\right).
\end{align}
\\
Define
\begin{align} \lab{eq:HW_fb}
\begin{aligned}
 H_k &= \{z \mid \langle z-y_k,v_k \rangle \leq \varepsilon_k\},\\[2mm]
 W_k &= \{z \mid \langle z-x_k,x_0 - x_k \rangle \leq 0\},
\end{aligned}
\end{align}
where
\begin{align} \lab{eq:ve_fb}
\begin{aligned}
 v_k &= \dfrac{\tiw_{k}-y_k}{\lambda},\\[2mm]
\varepsilon_k &= \dfrac{\norm{y_k - w'_{k}}^2}{4L^{-1}}.
\end{aligned}
\end{align}
\\
Set
\begin{align}\lab{eq:P_fb}
 x_{k+1} = P_{H_k \cap W_k}(x_0).
\end{align}
}
\end{algorithm}

\mgap
\mgap

Next we make some remarks regarding Algorithm \ref{alg:fb}:
\begin{enumerate}[label = (\roman*)]
\item The role and meaning of the extrapolations $w_k$ and $\widetilde w_k$ as in \eqref{eq:defw_fb} are the same as discussed before in the first remark following Algorithm \ref{alg:main}. The main step in Algorithm \ref{alg:fb} consists in the computation of the (forward-backward step) $y_k$ as in \eqref{eq:fb}, which resembles the forward-backward method 
(see \cite{LM79, Pas79}).  The projection $w'_k$ of $\widetilde w_k$ over $C$ is necessary to recover feasibility of 
$\widetilde w_k$ with respect to (the feasible set) $C$. Note that of $C = \Ha$, i.e., if $F(\cdot)$ is defined in the whole space, then
one can simply take $w'_k = \widetilde w_k$. Furthermore, the half-spaces $H_k$ and $W_k$ and the update $x_{k+1}$ as in 
\eqref{eq:HW_fb} and \eqref{eq:P_fb}, respectively, are defined similarly to those in the previous algorithms of this paper. 
\item The main advantage of Algorithm \ref{alg:fb} when compared to Algorithm \ref{alg:tseng} is that the first requires no extra evaluation of $F(\cdot)$ at $y_k$, while the former requires the computation of $F(y_k)$ in order to define $v_k$ (see \eqref{eq:itse02}). On the other hand, Algorithm \ref{alg:tseng} provides a potentially better approximate solution to \eqref{eq:mipt}, since the inclusion in \eqref{pervj03} provides a point (exactly) in the graph of $F(\cdot) + B(\cdot)$, while 
\eqref{pervj03x} below gives a point in its $\varepsilon$-enlargement. 
\item While Algorithm \ref{alg:fb} is similar to the main algorithm as proposed and studied in~\cite[Eq. (3.1)]{DJCS17}, we mention that contrary to our work, \cite{DJCS17} doesn't provide any complexity analysis for their method. 
\end{enumerate}

\mgap

We now show that Algorithm \ref{alg:fb} is also a special instance of Algorithm \ref{alg:main}.

\bprop \label{pr:fb.e.hpp}
For the sequences evolved by \emph{Algorithm \ref{alg:fb}}, the following hold for all $k\geq 0$:
 \begin{align} \lab{eq:fb02}
   \begin{aligned}
   & v_k\in (F^{\varepsilon_k}+B)(y_k)\subset (F+B)^{\varepsilon_k}(y_k),\\[2mm]
   & \lambda v_k + y_k - \tiw_{k}=0,\quad
   2\lambda \varepsilon_k\leq \sigma^2 \norm{y_k-\tiw_{k}}^2.
 \end{aligned}
\end{align}
As a consequence, it follows that \emph{Algorithm \ref{alg:fb}} is a special instance of 
\emph{Algorithm \ref{alg:main}} \emph{(}with $\lambda_k\equiv \lambda$\emph{)} for solving \eqref{eq:mipt} with $T = F + B$.
\eprop
\begin{proof}
The proof follows the same outline of proof of~\cite[Proposition 5.3]{Sva14}.
\end{proof}

\mgap

Next is the main result regarding the (strong) convergence and iteration-complexity of Algorithm~\ref{alg:fb}.

\begin{theorem}[Strong convergence and iteration-complexity of Algorithm \ref{alg:fb}] \label{th:fb.main}
 Suppose assumptions \emph{C1}, \emph{C2} and \emph{C3} and the $(1/L)$-cocoercivity condition \eqref{eq:f.coco} as above hold. Consider the sequences evolved by \emph{Algorithm \ref{alg:fb}} and let $d_0$ denote the distance of $x_0$ to the solution set  $\mathcal{S}\neq \emptyset$ of \eqref{eq:mipt}, and suppose \emph{Assumption \ref{ass:ab}} holds on the sequences $\set{\alpha_k}$ and $\set{\beta_k}$. 
Then the following statements hold:
\begin{enumerate}[label = \emph{(\alph*)}]
\item The sequences $\{x_k\}$ and $\set{y_k}$ converge strongly to $x^\ast \coloneqq P_{\mathcal{S}}(x_0)$.
\item For all $k\geq 0$, there exists $j\in [k]$ such that
\begin{subnumcases}{ }
v_j\in (F^{\varepsilon_j} + B)(y_j), & \label{pervj03x}\\[3mm]
\norm{v_j} \leq \dfrac{2 d_0 L}{\sqrt{k+1}}\left( \frac{\sqrt{(1 + \overline \alpha)\left[(1 + \overline \beta)\left( 1 + \overline \alpha (1 + \overline \beta) \right) + \overline s\right]}}{2(1- \sigma ^2)\sigma^2}\right), & \label{des:vc3x} \\[3mm]
\varepsilon_j \leq \frac{d_0^2 L}{k+1}\left(\frac{4 \sigma^2 (1 + \overline \alpha)\left[(1 + \overline \beta)\left( 1 + \overline \alpha (1 + \overline \beta) \right) + \overline s\right]}{ 2(1 -\sigma^2)^2\sigma^2}\right).& \label{des:ec0x} 
\end{subnumcases}
\end{enumerate}
\end{theorem}
\begin{proof}
The proof follows the same outline of Theorem \ref{th:korp}'s proof, now using Proposition \ref{pr:fb.e.hpp} instead of Proposition \ref{pr:korp_sp}.
\end{proof}

\section{Numerical experiments}
\lab{sec:ne}

In this section, we present preliminary numerical experiments for solving the matrix game problem
\begin{align} \lab{prob:mg}
\min_{x\in \Delta_n} \max_{y \in \Delta_m}\,\inner{x}{Ay},
\end{align}
where $ \Delta_n $ and $ \Delta_m $ denote the standard unit simplexes in $ \R^n $ and $ \R^m $, respectively, and $ A $ is a $ n \times m $ matrix.

\mgap

Problem \eqref{prob:mg} is equivalent to the monotone inclusion \eqref{eq:mipt} with
$ \Ha \coloneq \R^n \times \R^m $ (endowed with the standard inner product), and
$ F \colon \Ha \to \Ha $ and $ B \colon \HH \tos \HH $ given, respectively, by
\begin{align} \lab{eq:def.fbnum}
F(x, y) = (Ay, -A^\top x), \quad B(x, y) = N_{\Delta_n \times \Delta_m}(x, y)
\quad \text{for all} \enspace (x, y) \in \HH,
\end{align}
where $ N_{\Delta_n \times \Delta_m} $ denotes the normal cone of $ \Delta_n \times \Delta_m $. In this case, note that 
the sets $ \Dom F $ and $ C $ are both equal to $ \HH $ and $ \Dom B = \Delta_n \times \Delta_m $ (see conditions C1--C3 in Section \ref{sec:tseng}). Moreover, the map $ F $ as in \eqref{eq:def.fbnum} is clearly $ L $-Lipschitz continuous (see \eqref{eq:f.Lip}) with
$ L = \norm{A}_2 $. Also, note that the resolvent $ (\lambda B + I)^{-1} $ reduces to the orthogonal projection
$ P_ {\Delta_n \times \Delta_m} $ onto the closed and convex set $ \Delta_n \times \Delta_m $.

\mgap

We tested two algorithms for solving \eqref{prob:mg}:
\begin{itemize}
\item Algorithm \ref{alg:tseng} from this paper, which we call \emph{StrIneFBF}, with
$ \lambda = 0.37/\norm{A}_2 $, $ \beta_k \equiv 0 $\\ and
$ \alpha_k \equiv \alpha \in \set{0, 0.05, 0.10, 1, 2, 20, 50} $. 
\item The inertial variant of the Tseng's forward-backward method as given in~\cite[Algorithm 3]{AM20}, which we call here
\emph{IneFBF}, with $ \lambda_k \equiv \lambda = 0.37/\norm{A}_2 $, 
$ \alpha_k \equiv \alpha \in \set{0, 0.05, 0.10, 0.15, 0.25} $  
and $ \tau = 1 $ (in the notation of~\cite{AM20}). 
We observe that the theoretical analysis for IneFBF (see~\cite{AM20}) does not support choosing the inertial parameter 
$ \alpha_k \geq 0 $ larger than one.
\end{itemize}

The data matrix $ A $ was generated by using the Numpy commands \texttt{np.random.uniform(-1, 1, size=(n, n))}
and \texttt{np.random.seed(42)}.
We conducted the experiments in a Jupyter notebook executed on the Colab environment with Python 3.10.12
and used the following termination criterion:
\begin{align}
 \max\{\norm{x_{k+1} - x_k}, \norm{y_{k+1} - y_k}\} < 10^{-4}. 
\end{align}

\mgap

In Tables \ref{tb:01}, \ref{tb:02} and \ref{tb:03} we present the obtained numerical results for dimension $ n = 80 $, 
$ n = 150 $ and  $ n = 200 $, respectively. One can see that StrIneFBF (Algorithm \ref{alg:tseng} from this paper) outperforms IneFBF.

\mgap
\mgap

% TABLE 1
\begin{table}[htbp]
\centering
\setlength{\abovecaptionskip}{4pt} % espaço entre legenda e tabela
\setlength{\belowcaptionskip}{0pt} % espaço após a legenda
\caption{StrIneFBF vs. IneFBF ($ n=80 $).}
\lab{tb:01}
\small
\renewcommand{\arraystretch}{1}

\begin{tabularx}{0.92\textwidth}{l
  S[table-format=2.2,round-mode=places,round-precision=2]
  S[table-format=4.0]
  >{\centering\arraybackslash}X
  >{\centering\arraybackslash}X
  >{\centering\arraybackslash}X}
\toprule
Algorithm & {$ \alpha $} & {Iterations} &
{\footnotesize $ \norm{x_{k+1} - x_k} $} &
{\footnotesize $ \norm{z_k - P_C(z_k - F(z_k))} $} &
{\footnotesize $ \norm{x_k - P_C(x_k - F(x_k))} $} \\
\midrule
StrIneFBF  & 0.00   & 2616     & 0.000099 & 0.000056 & 0.000052 \\
StrIneFBF  & 0.05   & 940      & 0.000015 & 0.000042 & 0.000029 \\
StrIneFBF  & 0.10   & 979      & 0.000003 & 0.000055 & 0.000047 \\
%StrIneFBF  & 0.15   & 1031     & 0.000003 & 0.000070 & 0.000049 \\
%StrIneFBF  & 0.25   & 1076     & 0.000004 & 0.000075 & 0.000067 \\
StrIneFBF  & 1.00   & 3891     & 0.000100 & 0.000086 & 0.000061 \\
StrIneFBF  & 2.00   & 4        & 0.000000 & 0.005080 & 0.003064 \\
%StrIneFBF  & 5.00   & 2        & 0.000000 & 0.078967 & 0.023588 \\
StrIneFBF  & 20.00  & 2        & 0.000000 & 0.109228 & 0.031494 \\
StrIneFBF  & 50.00  & 2        & 0.000000 & 0.109228 & 0.031494 \\
IneFBF     & 0.00   & 3759     & 0.000100 & 0.000258 & 0.000159 \\
IneFBF     & 0.05   & 3274     & 0.000098 & 0.000245 & 0.000168 \\
IneFBF     & 0.10   & 2924     & 0.000100 & 0.000251 & 0.000181 \\
IneFBF     & 0.15   & 3510     & 0.000099 & 0.000249 & 0.000161 \\
IneFBF     & 0.25   & 4376     & 0.000099 & 0.000270 & 0.000158 \\
\bottomrule
\end{tabularx}
\label{tab:resultados-unif-n80}
\end{table}

\mgap
\mgap

% TABLE 2
\begin{table}[htbp]
\centering
\setlength{\abovecaptionskip}{4pt} % espaço entre legenda e tabela
\setlength{\belowcaptionskip}{0pt} % espaço após a legenda
\caption{StrIneFBF vs. IneFBF ($ n=150 $).}
\lab{tb:02}
\small
\renewcommand{\arraystretch}{1}

\begin{tabularx}{0.92\textwidth}{l
  S[table-format=2.2,round-mode=places,round-precision=2]
  S[table-format=4.0]
  >{\centering\arraybackslash}X
  >{\centering\arraybackslash}X
  >{\centering\arraybackslash}X}
\toprule
Algorithm & {$ \alpha $} & {Iterations} &
{\footnotesize $ \norm{x_{k+1} - x_k} $} &
{\footnotesize $ \norm{z_k - P_C(z_k - F(z_k))} $} &
{\footnotesize $ \norm{x_k - P_C(x_k - F(x_k))} $} \\
\midrule
StrIneFBF  & 0.00   & 266     & 0.000082 & 0.000429 & 0.000187 \\
StrIneFBF  & 0.05   & 231     & 0.000086 & 0.000077 & 0.000052 \\
StrIneFBF  & 0.10   & 356     & 0.000058 & 0.000064 & 0.000045 \\
%StrIneFBF  & 0.15   & 372     & 0.000068 & 0.000069 & 0.000048 \\
%StrIneFBF  & 0.25   & 393     & 0.000087 & 0.000080 & 0.000053 \\
StrIneFBF  & 1.00   & 105     & 0.000100 & 0.001661 & 0.022685 \\
StrIneFBF  & 2.00   & 3       & 0.000000 & 0.020903 & 0.013953 \\
%StrIneFBF  & 5.00   & 2       & 0.000000 & 0.020005 & 0.017358 \\
StrIneFBF  & 20.00  & 2       & 0.000000 & 0.019376 & 0.022685 \\
StrIneFBF  & 50.00  & 2       & 0.000000 & 0.019376 & 0.022685 \\
IneFBF     & 0.00   & 1541    & 0.000100 & 0.000163 & 0.000099 \\
IneFBF     & 0.05   & 1471    & 0.000098 & 0.000172 & 0.000093 \\
IneFBF     & 0.10   & 1404    & 0.000100 & 0.000186 & 0.000090 \\
IneFBF     & 0.15   & 2105    & 0.000099 & 0.000158 & 0.000084 \\
IneFBF     & 0.25   & 4628    & 0.000098 & 0.000143 & 0.000073 \\
\bottomrule
\end{tabularx}
\label{tab:resultados-unif-n150}
\end{table}

\mgap
\mgap

% TABLE 3
\begin{table}[htbp]
\centering
\setlength{\abovecaptionskip}{4pt} % espaço entre legenda e tabela
\setlength{\belowcaptionskip}{0pt} % espaço após a legenda
\caption{StrIneFBF vs. IneFBF ($ n=200 $).}
\lab{tb:03}
\small
\renewcommand{\arraystretch}{1}

\begin{tabularx}{0.92\textwidth}{l
  S[table-format=2.2,round-mode=places,round-precision=2]
  S[table-format=4.0]
  >{\centering\arraybackslash}X
  >{\centering\arraybackslash}X
  >{\centering\arraybackslash}X}
\toprule
Algorithm & {$ \alpha $} & {Iterations} &
{\footnotesize $\norm{x_{k+1} - x_k} $} &
{\footnotesize $\norm{z_k - P_C(z_k - F(z_k))}$} &
{\footnotesize $\norm{x_k - P_C(x_k - F(x_k))}$} \\
\midrule
StrIneFBF  & 0.00   & 216     & 0.000087 & 0.000339 & 0.000186 \\
StrIneFBF  & 0.05   & 193     & 0.000081 & 0.000412 & 0.000208 \\
StrIneFBF  & 0.10   & 283     & 0.000097 & 0.000167 & 0.000092 \\
%StrIneFBF  & 0.15   & 303     & 0.000097 & 0.000144 & 0.000080 \\
%StrIneFBF  & 0.25   & 271     & 0.000099 & 0.000245 & 0.000136 \\
StrIneFBF  & 1.00   & 42      & 0.000093 & 0.001730 & 0.000903 \\
StrIneFBF  & 2.00   & 3       & 0.000000 & 0.015022 & 0.009224 \\
%StrIneFBF  & 5.00   & 2       & 0.000000 & 0.045028 & 0.016056 \\
StrIneFBF  & 20.00  & 2       & 0.000000 & 0.045028 & 0.016056 \\
StrIneFBF  & 50.00  & 2       & 0.000000 & 0.045028 & 0.016056 \\
IneFBF     & 0.00   & 1161    & 0.000100 & 0.000146 & 0.000090 \\
IneFBF     & 0.05   & 956     & 0.000099 & 0.000160 & 0.000094 \\
IneFBF     & 0.10   & 904     & 0.000100 & 0.000163 & 0.000100 \\
IneFBF     & 0.15   & 1397    & 0.000099 & 0.000144 & 0.000079 \\
IneFBF     & 0.25   & 2303    & 0.000098 & 0.000146 & 0.000073 \\
\bottomrule
\end{tabularx}
\label{tab:resultados-unif-n200}
\end{table}

\mgap
\mgap

% Figure 1
\begin{figure}[!htbp]
  \centering
  \caption{Comparative residual evolution: StrIneFBF vs. IneFBF.}
  \label{fig:residual}
    \includegraphics[width=0.9\linewidth]{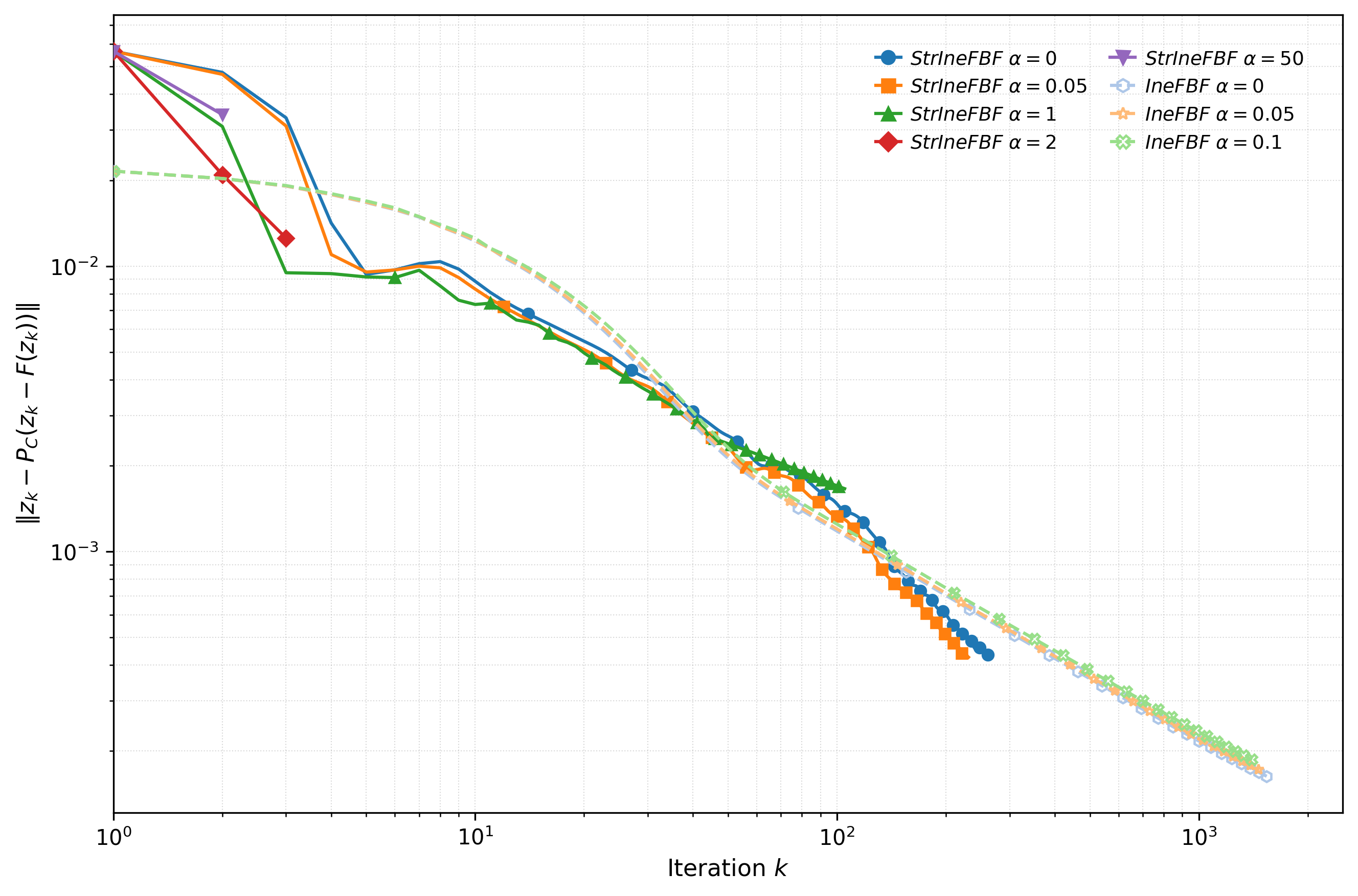} 
\end{figure}

%\newpage
\appendix

\section{Well-definedness of Algorithm \ref{alg:main}} \lab{subsec:wd}

In this appendix, we will prove the well-definedness of Algorithm \ref{alg:main} (note that $x_{k+1}$ as in \eqref{eq:wood04} depends on the nonemptyness of the set $H_k\cap W_k$). 
Recall that we are assuming that the solution $\mathcal{S} \coloneqq T^{-1}(0)$ of \eqref{eq:probm} is nonempty.
The proof of the following proposition follows the same outline of~\cite[Proposition 3]{SS00}.

\bprop \label{pr:well_defined}
Suppose \emph{Algorithm \ref{alg:main}} reaches the iteration $k$ and $x_k \in M_{\mathcal{S}}(x_0)$, where
\begin{align} \lab{def:msx}
M_{\mathcal{S}}(x_0)  \coloneqq \Set{ x \mid \inner{z-x}{x_0 - x} \leq 0 \quad \mbox{for all}\;\;  z \in \mathcal{S}}.
\end{align}
Then  the following statements hold:
\begin{enumerate}[label = \emph{(\alph*)}]
\item \lab{well_defined:seg} 
$\mathcal{S} \subset H_k \cap W_k$.
\item \lab{well_defined:ter}
$x_{k+1}$ is well-defined and $x_{k+1} \in M_{\mathcal{S}}(x_0)$.
\end{enumerate}
\eprop
\begin{proof} 
\emph{\ref{well_defined:seg}}  Since Algorithm \ref{alg:main} reaches iteration $k$, it follows that $H_k$ and $W_k$ as in \eqref{eq:wood02}
are well-defined. Moreover, using the inclusion in \eqref{eq:wood} and the definition of $T^\varepsilon$ as in \eqref{def:enlar} we conclude that $\mathcal{S} \subset H_k$. 
Since, by assumption, $x_k \in M_{\mathcal{S}}(x_0)$, it follows from \eqref{def:msx} that 
\[
\inner{z - x_{k}}{x_0 - x_k} \leq 0 \quad \mbox{for all}\quad z \in \mathcal{S},
\]
which is to say that $\mathcal{S}\subset W_k$ (see \eqref{eq:wood02}). Therefore, $\mathcal{S} \subset H_k \cap W_k$.

\mgap

\emph{\ref{well_defined:ter}} Since $\mathcal{S}\neq \emptyset$, from item (a) we obtain $H_k \cap W_k \neq \emptyset$, and so the next iterate $x_{k+1}$ is well-defined ($H_k \cap W_k$ is trivially closed and convex).
 Using the fact that $x_{k+1}$ is the projection of $x_0$ onto $H_k \cap W_k$, from \eqref{projetion-caracterization} we have
\[
\inner{z - x_{k+1}}{x_0 - x_{k+1}} \leq 0 \quad \mbox{for all}\quad  z \in H_k \cap W_k. 
\]
As $\mathcal{S} \subset H_k \cap W_k$, the above inequality then holds for all $z \in \mathcal{S}$, implying that 
$x_{k+1}\in M_{\mathcal{S}}(x_0)$. This finishes the proof of the proposition.
\end{proof}

As a consequence of the proposition above, we have that the whole algorithm is well-defined.

\bcoro[{Well-definedness of Algorithm \ref{alg:main}}]\label{cor:2}
We have that \emph{Algorithm 1} is well-defined and generates ``infinite'' sequences. 
Furthermore,  
\begin{align} \lab{eq:sshk}
 \mathcal{S} \subset H_k \cap W_k\quad \mbox{for all}\quad k\geq 0,
\end{align}
where $\mathcal{S} \coloneqq T^{-1}(0)\neq \emptyset$ is the solution set of \eqref{eq:probm}. 
\ecoro
\begin{proof}
Note that $x_0 \in M_{\mathcal{S}}(x_0)$, apply induction on $k\geq 0$, and use Proposition \ref{pr:well_defined}.
\end{proof}

\section{Auxiliary results} \lab{app:aux}

\begin{lemma}\label{lm:imp01}
Let $\Ha$ be a real Hilbert space. The following statement holds:
\begin{enumerate}[label = \emph{(\alph*)}]
\item \lab{imp01:seg} 
For any $x,y\in \Ha$ and $t\in \mathbb{R}$, we have
\begin{align*}
\norm{tx + (1 - t)y}^2 = t\norm{x}^2 + (1 - t)\norm{y}^2 - t(1 - t)\norm{x-y}^2.
\end{align*}
\item \lab{imp01:ter}
For any $x, y, z\in \Ha$, we have
\begin{align*}
2\langle x - y, x - z\rangle =  \|x-y\|^2 + \|x - z\|^2 - \|y-z\|^2.
\end{align*}
\item \lab{imp01:qua}
For every sequence $\set{x_n}$ and $x$ in $\Ha$, we have
\begin{align*}
x_n \to x\quad \mbox{if and only if}\quad  x_n\wto x\;\; \mbox{and}\;\; \norm{x_n}\to \norm{x}.
\end{align*}
\end{enumerate}
\end{lemma}

\begin{lemma}\label{lm:min}
 Given $a,b\geq 0$ with $a + b > 0$ and $c\in \mathbb{R}_+$, we have
 \begin{align*}
  \min\{a s^2 + b t^2  \mid s,t\geq 0\;\;\mbox{and}\;\; s + t \geq c\} = \frac{ab}{a + b} c^2.
 \end{align*}
\end{lemma}

%\bibliographystyle{plain}
%\bibliography{prox}

\def\cprime{$'$} \def\cprime{$'$}

\end{document}